%% file: Robust_v6_extended.tex
\documentclass{aptpub}

\authornames{Ramtin Pedarsani, Jean Walrand, Yuan Zhong} 
\shorttitle{Robust Scheduling} 

\usepackage{graphicx,epsfig,graphics,epsf}
\usepackage{comment, color, caption, subcaption}
\usepackage{soul}
\usepackage{footnote}

\newcommand{\yuan}[1]{{\color{black}#1}}
\newcommand{\ramtin}[1]{{\color{black}#1}}

\newcommand{\yz}[1]{{\color{black}#1}}

\newcommand{\veps}{\varepsilon}
\newcommand{\eps}{\epsilon}
\newcommand{\EE}{\mathbb{E}}
\newcommand{\PP}{\mathbb P}

\newcommand{\wE}{\widetilde{E}}
\newcommand{\bE}{\bar{E}}
\newcommand{\bOne}{\boldsymbol{1}}
\newcommand{\tB}{\tilde{B}}
\newcommand{\tE}{\tilde{E}}

\begin{document}

\title{Robust Scheduling for Flexible \\Processing Networks} 

\authorone[University of California, Santa Barbara]{Ramtin Pedarsani} 
\authortwo[University of California, Berkeley]{Jean Walrand}
\authorthree[Columbia University]{Yuan Zhong}

\addressone{Department of Electrical and Computer Engineering, UC Santa Barbara, Santa Barbara, CA 93106, USA. Email: ramtin@ece.ucsb.edu}

\addresstwo{257 Cory Hall, Department of Electrical Engineering and Computer Sciences, UC Berkeley, Berkeley, CA 94720, USA. Email: walrand@berkeley.edu}

\addressthree{500 W. 120th St., Mudd 344, New York, NY 10027, USA. Email: yz2561@columbia.edu}


\begin{abstract}
\yuan{Modern processing networks often consist of 
heterogeneous servers with widely varying capabilities, 
and process job flows with complex structure and requirements. 
A major challenge in designing efficient scheduling policies 
in these networks is the lack of reliable estimates of system parameters, 
and an attractive approach for addressing this challenge is to design robust policies, 
i.e., policies that do not use system parameters such as arrival and/or service rates 
for making scheduling decisions. 

In this paper, we propose a general framework for the design of robust policies. 
The main technical novelty is the use of a stochastic gradient projection method 
that reacts to queue-length changes in order to find a balanced allocation of service 
resources to incoming tasks. 
We illustrate our approach on two broad classes of processing systems, 
namely the flexible fork-join networks and the flexible queueing networks, 
and prove the rate stability of our proposed policies for these networks 
under non-restrictive assumptions. 
} \vspace{0.2in}
\end{abstract}

\keywords{Robust Scheduling; Flexible Queueing Networks; Stochastic Gradient Projection} 
\ams{60K25}{90B15; 60G17}

\renewcommand{\thefootnote}{\arabic{footnote}}
\setcounter{footnote}{0}

\section{Introduction}\label{sec:intro}
As modern processing systems 
(e.g., data centers, hospitals, manufacturing networks) 
grow in size and sophistication, their infrastructures become more complicated, 
and a key operational challenge in many such systems 
is the efficient scheduling of processing resources to meet 
various demands in a timely fashion. 
\yuan{A scheduling policy decides how server capacities are allocated over time, 
and a major challenge in designing such policies is the lack of knowledge of system parameters due to the complex processing environment and the volatility of the jobs to be processed. 
Demands are diverse, typically unpredictable, and can occur in bursts; 
furthermore, operating conditions of processing resources 
can vary over time (see e.g., \cite{DCtraffic}). Thus, estimates of key system parameters such as arrival and/or service rates  
are often unreliable, and can frequently become obsolete.} One approach to \yuan{address} this complicated scheduling and resource allocation problem 
is to design \emph{robust} scheduling policies, 
\yuan{where scheduling decisions are made only based on current and/or past system states such as 
queue sizes, and not depending on system parameters such as arrival or service rates.}
Robust scheduling policies \yuan{can be} highly desirable in practice, 
since 
they use only minimal information and \yuan{can} adapt 
to changes in demands and service conditions automatically. The main objective of this paper is to develop 
\yuan{a general framework} for designing robust policies and analyzing their performance.

\yuan{Consider a single-server queueing system 
with unit-size jobs arriving at an unknown rate $\lambda$, 
and a server with a costly and sufficiently large service capacity $\mu$ (in particular, $\mu > \lambda$), whose precise value is unknown.
Suppose that at regular time intervals, 
the server can adjust its service effort, 
measured by the fraction $p \in [0, 1]$ of the total capacity 
(we can implement this in practice as a randomized decision 
of serving the queue with probability $p$). 
The goal is to keep the system stable. 
Let $\Delta Q$ be the queue size change over a regular time interval. 
Intuitively, if $\Delta Q > 0$, then it is likely that the arrival rate 
is faster than the dedicated service effort, which should then be increased. If $\Delta Q < 0$, 
then the service effort should be decreased for cost consideration. 
This naturally leads to an update rule for the service effort from time $n$ to $n+1$ of the form $p^{n+1}:=p^n + \beta^n \Delta Q$ 
with $\beta^n > 0$. Under mild technical conditions on the sequence $\{\beta^n\}$, it can be shown that 
$\mu^n \to \lambda$ almost surely, implying system stability.}

\yuan{This simple example illustrates the high-level approach of our policy design framework: 
allocate more (less) service to a queue if the corresponding queue size increases (decreases). 
Our design approach only uses the system state information -- namely, the queue size changes -- 
and does not require information on either the arrival or service rates. 
A simple but key observation that justifiess the validity of this approach is that 
if the queue size at the start of an interval is sufficiently large, 
then $\Delta Q$, the queue size change, is proportional to \yz{$\lambda - \mu p$} in expectation. 
In a network setting, by building upon this simple observation, we can decide 
how to allocate shared resources among competing queues based on their respective queue size changes.} 

Our methodology is general and can be applied to \yuan{a wide range of} processing networks. 
\yuan{To illustrate our approach concretely}, we focus on two \yuan{broad classes of} network models, 
which \yuan{(a) generalize many important classes of queueing network models, 
such as parallel server systems \cite{MS2004} and fork-join networks \cite{Nguyen1993} (see Section \ref{ssec:related} for more details), 
and (b) model key features of dynamic resource allocation at fine granularity in many} 
modern applications such as cloud computing, \yuan{flexible manufacturing}, and large-scale healthcare systems. \yuan{We now proceed to describe our network models and contributions in more detail.}

\yuan{Common to many modern large-scale processing systems are the following two important features:} (a) workflows of interdependent tasks, where the completion of one task produces new tasks to be processed in the system, and b) flexibility of processing resources with overlapping capabilities as well as flexibility of tasks to be processed by multiple servers. 
To illustrate these two features, consider 
the scheduling of a simple Mapreduce job \cite{mapreduce}
of word count of the play ``Hamlet'' in a data center (see Figure \ref{fig:hamlet}). 
{\em ``Mappers''} are assigned the tasks of word count by Act, 
producing intermediate results, which are then aggregated by the ``reducer''. 
In more elaborate workflows, these interdependencies 
can be more complicated. 
\yuan{There may also be} considerable overlap in the processing capabilities 
of the data center servers, and flexibility on where tasks can be placed \cite{mapreduce}.
Similarly, in a healthcare facility such as a hospital, 
an arriving patient may have a complicated workflow of service/treatment requirements \cite{AMZ2012}, which can also be assigned to doctors and/or nurses with overlapping capabilities. 
\yuan{To capture the dependencies in workflows and the system flexibility, we consider the following two classes of processing networks:}

\begin{itemize}
\item[(i)] a flexible \yuan{fork-join} processing 
network model, in which jobs are modeled as directed acyclic 
graphs (DAG), with nodes representing \emph{tasks}, and 
edges representing \emph{precedence constraints} among tasks,  
\yuan{and servers have overlapping capabilities; and}
\item [(ii)] a flexible queueing network with probabilistic routing structure, where a job goes through processing steps in different queues \yuan{according to a routing matrix, and servers have overlapping capabilities}. 
\end{itemize}

We design a robust scheduling policy for \yuan{each class of networks}, 
\yuan{and analyze performance properties of the proposed policies.}
\yuan{For both models,} we prove the throughput optimality\footnote{We are concerned with {\em rate stability} in this paper. A scheduling policy is {\em throughput optimal} if, under this policy, the system is stable whenever there exists some policy under which the system is stable.} of our policies, under \yuan{a factorization criterion on service rates (see Assumption \ref{asmp:factor} of Section \ref{sec:policy} for details)}.
Our policy \yuan{design} is based on the simple idea of matching incoming flow rates to their respective service rates, and detecting mismatches using queue size information. If system parameters were known, a so-called static planning problem \cite{Harrison2000} can be solved to obtain the optimal allocation of server capacities, which balances flows in the system. Without the knowledge of system parameters, however, the policy updates the allocation of server capacities according to changes in queue sizes. 
\yuan{Methodologically,} our policy uses the idea of stochastic gradient descent (see e.g., \cite{Borkar2008}), a technique that \yuan{has been applied in the design of distributed policies in ad-hoc wireless networks \cite{JW2010}.}

\begin{figure}
\centering
\includegraphics[scale=0.3]{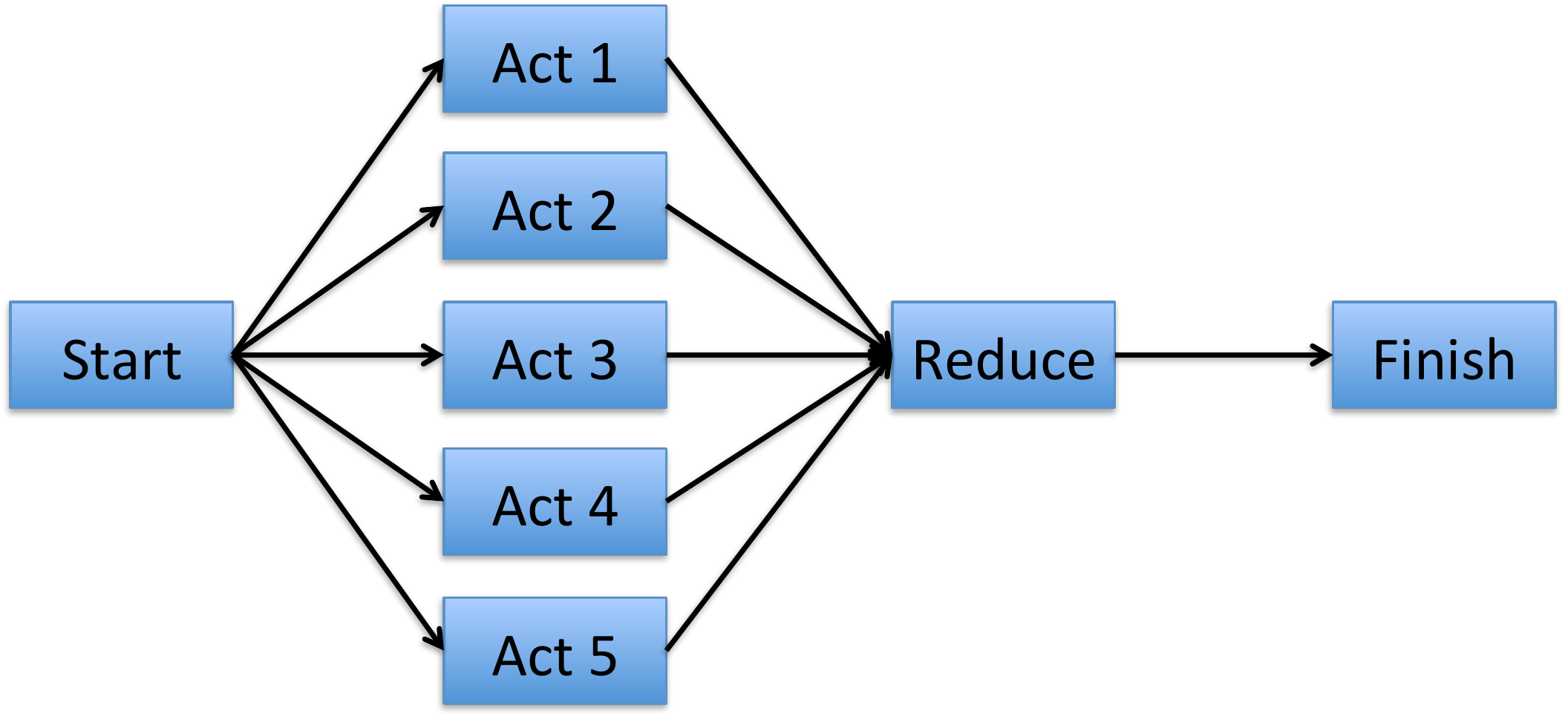}
\caption{Word count of Hamlet in MapReduce}
\label{fig:hamlet}
\end{figure}

\subsection{Related Works}\label{ssec:related} 
Scheduling of queueing networks has been studied extensively over several decades. We do not attempt
to provide a comprehensive literature review here; instead \yuan{we review the most relevant works.}

Our \yuan{flexible fork-join network} model 
is closely related to \yuan{and substantially generalizes the} classical fork-join networks 
(see e.g., \cite{BMS1989, BM1985, KW1989, BW1991, Nguyen1993, Nguyen1994}). 
The main difference between the classical models and ours is 
that we allow tasks to be flexible, whereas tasks are assigned to 
dedicated servers in classical fork-join networks. 
\yuan{In classical networks, simple robust policies such as FIFO (First-In-First-Out) 
can often shown to be throughput optimal, but need not be so in our flexible networks.}

The flexible queueing network model is closely related to the system considered in \cite{AAD2003} (\cite{AAD2003} also considers setup costs whereas we do not). 
The policies in \cite{AAD2003} make use of arrival and service rates, 
and their throughput properties are analyzed using fluid models, 
hence their approach is distinct from ours. 
\yuan{We would also like to point out that }in the case where the queues are not flexible, 
i.e., each queue has a dedicated server, the system reduces to 
the open multiclass queueing network (see e.g., \cite{HN93, dai99}). 

\yuan{Both the flexible fork-join network model and the flexible queueing network model 
can be viewed as generalizations of the classical parallel server system, considered in e.g., \cite{MS2004}. 
The flexible fork-join network extends the parallel server system 
by allowing jobs to consist of tasks with precedence constraints, and the flexible queueing network 
extends the parallel server system by allowing probabilistic routing among jobs.} 
There is considerable interest in the study of robust scheduling algorithms 
in the context of parallel server systems. The well-know $Gc\mu$ rule 
-- equivalent to a MaxWeight policy with appropriately chosen weights on queues -- 
has been proved to have good performance properties, including throughput optimality
(e.g., \cite{MS2004}). The $Gc\mu$ rule does not make use of the knowledge of arrival rates, 
but does require the knowledge of service rates.  
\cite{BT2011} studies performance properties of the Longest-Queue-First (LQF) policy, 
which is robust to both arrival and service rates, 
and establishes its throughput optimality 
when the so-called activity graph is a tree.  
\cite{SY2012} established the throughput optimality of a priority discipline 
in a many-server parallel server system, which consists of server pools, each of which in turn consists of a large number of identical servers, also under the condition that the activity graph is a tree.
\cite{DW2006} established the throughput optimality of LQF under a local pooling condition. 
\yuan{To the best of our knowledge, 
in the flexible fork-join networks and the flexible queueing networks, both extensions of parallel server systems, 
which include precedence constraints and routing, respectively, the problem of designing robust scheduling policies 
has not been addressed prior to this work.}

\yuan{As mentioned earlier, the analysis of our policies uses the technique of stochastic gradient descent \cite{Borkar2008}, 
which has been successfully employed in the design 
of distributed CSMA algorithms for wireless networks \cite{JW2010}. 
Our analysis is different from that of the CSMA algorithms in several ways; 
for example, CSMA algorithms actively attempt to estimate the arrival and service rates, 
whereas our policy is adaptive, and only reacts to these parameters through queue size changes.}

\subsection{Paper Organization}\label{sec:organization}
The rest of the paper is organized as follows. In Section \ref{sec:DAG}, we introduce \yuan{the flexible fork-join network model}. We propose our robust scheduling policy, and state our main theorems.   
In Section \ref{sec:FQN}, we describe the flexible queueing network model, and design a robust scheduling policy for this network. We conclude the paper in Section \ref{sec:conclusion}. All proofs are provided in the appendices.

\section{Scheduling DAGs with Flexible Servers}\label{sec:DAG}
\subsection{System Model}\label{sec:system}
We consider a general flexible fork-join processing network, in which jobs are modeled as directed acyclic graphs (DAG). Jobs arrive to the system as a set of tasks, among which there are precedence constraints. Each node of the DAG represents one task type\footnote{{We 
make use of both the concepts of {\em tasks} and {\em task types}. 
To avoid confusion and overburdening terminology, 
we will \yz{often} use {\em node} synonymously with {\em task type} for the rest of the paper.}}, 
and each (directed) edge of the DAG represents a precedence constraint. 
More specifically, we consider $M$ classes of jobs, each of them represented by one DAG structure. 
Let ${\cal G}_m = ({\cal V}_m,{\cal E}_m)$ be the graph representing the job of class $m$, $~1\leq m \leq M$, where ${\cal V}_m$ denotes the set of nodes of type-$m$ jobs, and ${\cal E}_m$ the set of edges of the graph. Note that sets ${\cal V}_m$ are disjoint. Let ${\cal V} = \cup_{m=1}^M {\cal V}_m$ and ${\cal E} = \cup_{m=1}^M {\cal E}_m$. We suppose that each ${\cal G}_m$ is connected, 
so that there is an undirected path between any two nodes of ${\cal G}_m$. 
There is no directed cycle in any ${\cal G}_m$ by the definition of DAG. 
Let the number of nodes of job type $m$ be $K_m$, i.e. $|{\cal V}_m|=K_m$. 
Let the total number of nodes in the network be $K$. Thus, $\sum_{m=1}^M K_m = K$. We index the task types in the system by $k, ~1 \leq k \leq K$, starting from job type $1$ to $M$. 
Thus, task type $k$ belongs to job type $m(k)$ if 
$$\sum_{m'=1}^{m(k)-1} K_{m'}<k \leq \sum_{m'=1}^{m(k)} K_{m'}.$$ 
We call node $k'$ a {\em parent} of node $k$, 
if they belong to same job type $m$, and $(k',k) \in {\cal E}_m$. Let ${\cal P}_k$ denote the set of parents of node $k$. In order to start processing a type-$k$ task, 
the processing of all tasks of its parents {\em within the same job} should be completed. Node $k$ is said to be a \emph{root} of DAG type $m(k)$, if ${\cal P}_k = \emptyset$. We call $k'$ an \emph{ancestor} of $k$ if they belong to the same DAG, and there exists a directed path of edges from $k'$ to $k$. Let $L_k$ be the length of the longest path from the root nodes of the DAG, ${\cal G}_{m(k)}$, to node $k$. If $k$ is a root node, then $L_k = 0$.

There are $J$ servers in the processing network. A server is {\it flexible} if it can serve more than one type of tasks. A task type is flexible if it can be served by more than one server. In other words, servers can have overlap of capabilities in processing a node. For each $j$, we define ${\cal T}_j$ to be the set of nodes that server $j$ is capable of serving. Let $T_j = | {\cal T}_j |$. For each $k$, let ${\cal S}_k$ be the set of servers 
that can serve node $k$, and let $S_k = |{\cal S}_k |$. Without loss of generality, we also assume that $T_j, S_k \geq 1$ for all $j$ and $k$, so that each server can serve at least one node, and each node can be served by at least one server. 
\begin{example}
Figure \ref{fig:DAG} illustrates the DAG of one job type that consists of four
nodes $\{1, 2, 3, 4\}$.  There are two servers $1$ and $2$.  Server $1$ can process
tasks of types in the set ${\cal T}_1 = \{1, 2, 3\}$ and server $2$ can process tasks of types 
in the set ${\cal T}_2 = \{3, 4\}$.
When a type-$1$ task is completed, it ``produces'' one type-$2$ task and one type-$3$ task, 
both of which have to be completed before 
the processing of the type-$4$ task of the same job can start. 
\end{example}


\begin{figure}
        \centering
        \begin{subfigure}{0.45\textwidth}
                \includegraphics[width= 4.5cm, height = 3cm]{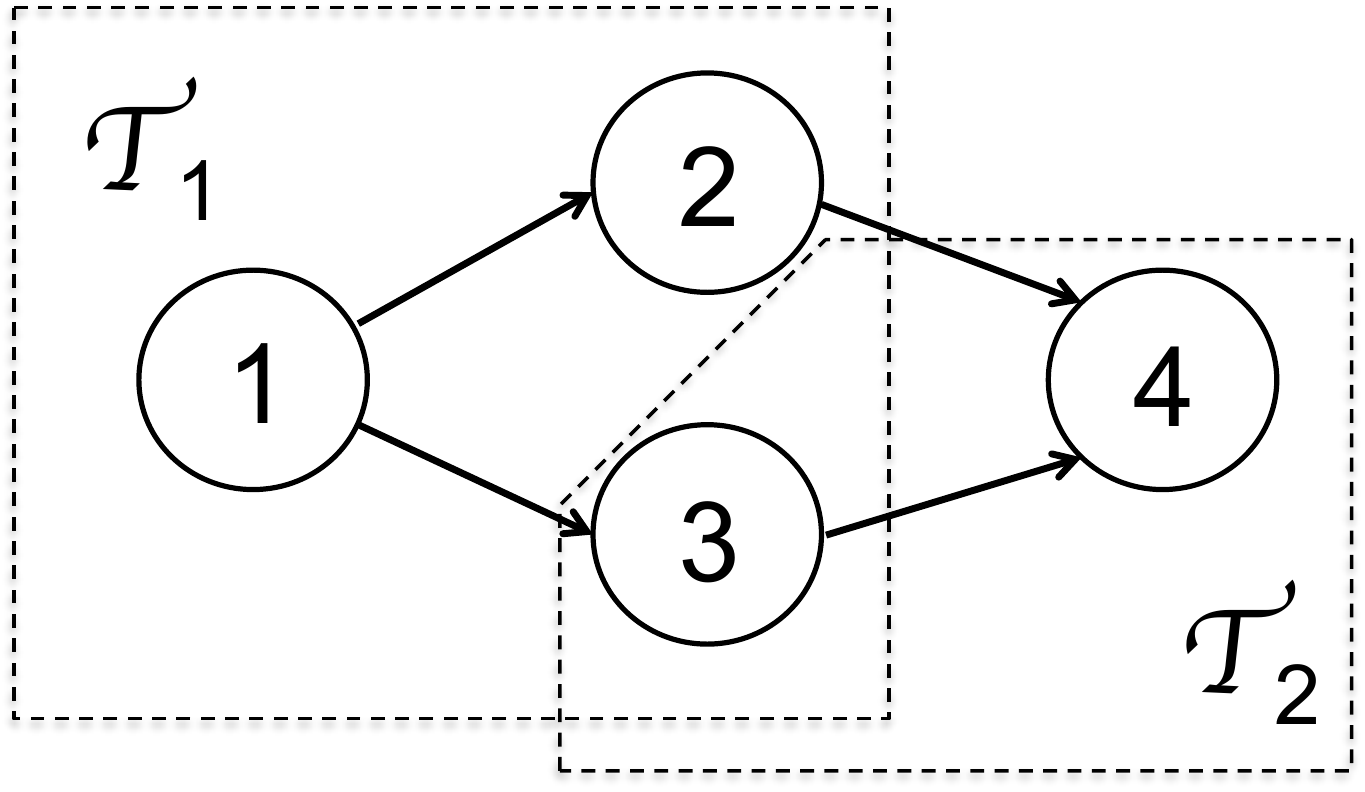}
  				\caption{\label{fig:DAG}A simple DAG} 
        \end{subfigure}
        \qquad
        \begin{subfigure}{0.45\textwidth}
               \includegraphics[width= 5cm, height = 3cm]{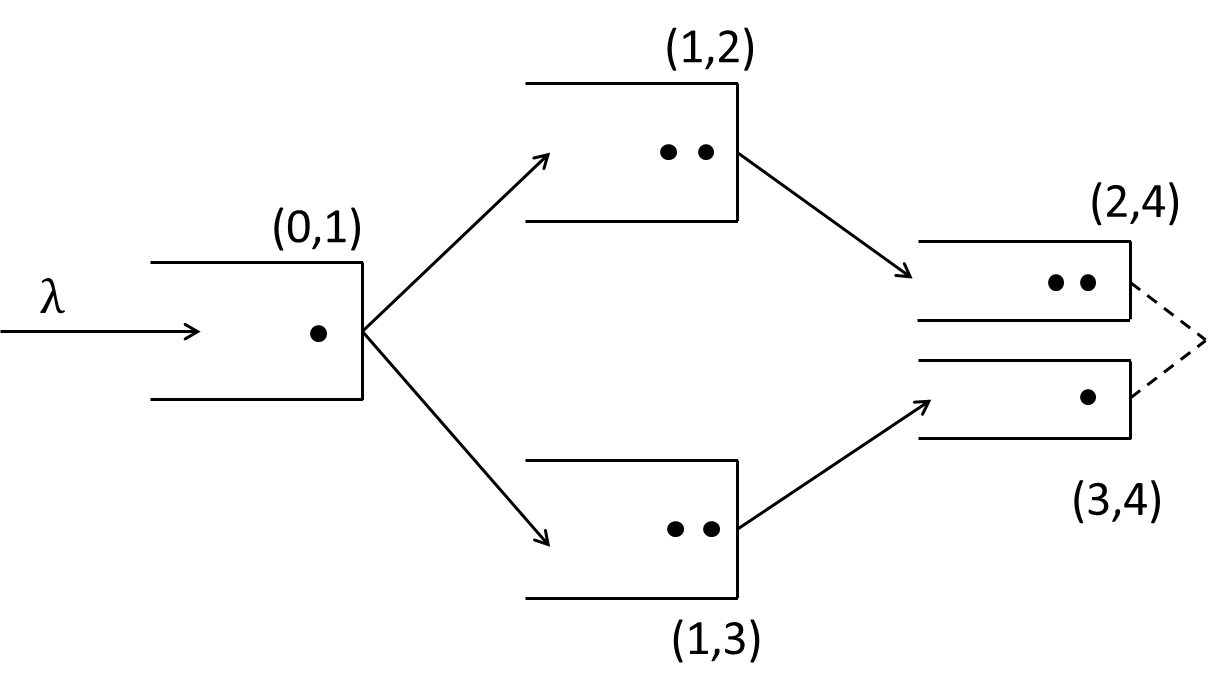}
  			\caption{\label{fig:QN}Queueing Network of the DAG}
        \end{subfigure}
        \caption{This figure illustrates a simple DAG and its corresponding queueing network.}\label{fig:DAG-QN}
\end{figure}

We consider the system in discrete time. We assume that the arrival process of type-$m$ jobs is an independent Bernoulli process with rate $\lambda_m$, $0 < \lambda_m <1$; that is, in each time slot, a new job of type $m$ arrives to the system with probability $\lambda_m$, independently over time. We assume that the service times are geometrically distributed and independent of everything else. When server $j$ processes task $k$, the service completion time has mean $\mu^{-1}_{kj}$. Thus, $\mu_{kj}$ can be interpreted as the service rate of node $k$ when processed by server $j$. 

\subsection{Queueing Network Model for Cooperative Servers}\label{sec:QN}
We model our processing system as a queueing network in the following manner. We maintain one virtual queue of processed tasks that are sent from node $k'$ to $k$ for each edge of the DAGs $(k',k) \in {\cal E}$. Furthermore, we maintain a virtual queue for the root nodes of the DAGs. Let $\chi_m$ be the number of root nodes in the graph of job type $m$. Then, the queueing network has $\sum_{m=1}^M (|{\cal E}_m| + \chi_m)$ virtual queues. As an example, consider the DAG of Figure \ref{fig:DAG}. \yuan{The virtual queues corresponding to} this DAG is \yuan{illustrated} in Figure \ref{fig:QN}. \yz{There are $5$ virtual queues in total, 4 for edges of the graph, and one for the root node $1$.}
\\ \\
\textbf{Job identities \yuan{and synchronization}.} In our model, jobs and tasks have distinct identities. 
This is mainly motivated by \yuan{applications to} data centers and healthcare systems. For instance, it is important not to mix up blood samples of different patients in hospital, and to put pictures on the correct webpage in a data center setting.

To illustrate how a job is processed in the preceding queueing network, 
consider the network in Figure \ref{fig:QN}. 
Suppose that a job of identity $a$ and with a DAG structure of Figure \ref{fig:DAG} enters the network. 
When task 1 of job $a$ from queue $(0, 1)$ is processed, tasks $2$ and $3$ of job $a$ are sent to queues $(1,2)$ and $(1,3)$, respectively. When tasks in queues $(1,2)$ and $(1,3)$ are processed, their results are sent to queues $(2,4)$ and $(3,4)$, respectively. Finally, to process task $4$ of job $a$, one task belonging to job $a$ from queue $(2,4)$ and one task belonging to job $a$ from $(3,4)$ are gathered and processed to finish processing job $a$. We emphasize that tasks are identity-aware in the sense that to complete processing task $4$, it is not possible to merge any two tasks (of possibly different jobs) from queues $(2,4)$ and $(3,4)$. 

\yuan{A common and important problem that needs to be addressed in scheduling DAGs 
is {\em synchronization}, where all parents of a task need to be completed 
for the task to be processed. In the presence of flexibility, 
synchronization constraints may lead to disorder in task processing, which adds synchronization penalty to the system; 
see \cite{PWZ2014} for an example.}
In this paper, to guarantee synchronization, we \yuan{assume the simplifying condition} that servers are cooperative \yuan{(this is equivalent to the case of cooperating servers 
described in \cite{AAD2003})}. That is, we assume that servers that work on the same task type, cooperate on the same head-of-the-line task, adding their service capacities. 
\yuan{In this way, tasks are processed in a FIFO manner so that no synchronization penalty is incurred.}  \\ \\
{\bf Queue Dynamics.} Now we describe the dynamics of the queueing network. Let $Q_{(k',k)}$ denote the length of the queue corresponding to edge $(k',k)$ and let $Q_{(0,k)}$ denote the length of the queue corresponding to root node $k$. A task of type $k$ can be processed 
if and only if $Q_{(k',k)} > 0$ for all $k' \in {\cal P}_k$ {-- this is because servers 
are cooperative, and tasks are processed in a FIFO manner}. Thus, the number of tasks of node $k$ available to be processed is $\min_{k' \in {\cal P}_k} Q_{(k',k)}$, if $k$ is not a root node, and $Q_{(0,k)}$, if $k$ is a root node. For example, in Figure \ref{fig:QN}, queue $(2,4)$ has length 2 and queue $(3,4)$ has length 1, so there is one task of type $4$ available for processing. When one task of class $k$ is processed, lengths of all queues $(k',k)$ are decreased by $1$, where $k' \in {\cal P}_k$, 
and lengths of all queues $(k,i)$ are increased by $1$, where $k \in {\cal P}_i$. Therefore, the dynamics of the queueing network is as follows. Let $d^n_k \in \{ 0,1\}$ be the number of processed tasks of type $k$ at time $n$, and $a^n_m \in \{0 , 1\}$ be the number of jobs of type $m$ that arrives at time $n$. If $k$ is a root node of the DAG, then
\begin{equation}\label{eq:dynamics1}
Q^{n+1}_{(0,k)} = Q^n_{(0,k)} + a^n_{m(k)} - d^n_k;
\end{equation}
else,
\begin{equation}\label{eq:dynamics2}
Q^{n+1}_{(k',k)} = Q^n_{(k',k)} + d^n_{k'} - d^n_k.
\end{equation}
 
Let $p_{kj}$ be the fraction of capacity that server $j$ allocates for processing available tasks of class $k$. We define $p=[p_{kj}]$ to be the \emph{allocation vector}. If server $j$ allocates all its capacity to different tasks, then $\sum_{k \in {\cal T}_j} p_{kj}=1$. Thus, an allocation vector $p$ is called \emph{feasible} if 
\begin{equation}\label{eq:feasible} 
\sum_{k \in {\cal T}_j} p_{kj} \leq 1, ~ \forall ~ 1 \leq j \leq J.
\end{equation}

We interpret the allocation vector at time $n$, $p^n = [p^n_{kj}]$, as randomized scheduling decisions at time $n$, in the following manner. 
First, without loss of generality, the system parameters can always be re-scaled 
so that \yz{$\sum_{j \in \mathcal{S}_k} \mu_{kj} < 1$} for all $k$, 
by speeding up the clock of the system.
Now suppose that at time slot $n$, 
the allocation vector is $p^n$.  
Then, the head-of-the-line task $k$ 
is served with probability $\sum_{j \in {\cal S}_k} \mu_{kj} p^n_{kj}$ in that time slot. 
Note that \yz{$\sum_{j\in {\cal S}_k} \mu_{kj} p^n_{kj} < 1$} 
by our scaling of the service rates. 

\subsection{The Static Planning Problem}\label{sec:static}
In this subsection, we introduce a linear program (LP) that characterizes the \emph{capacity region} of 
the network, defined to be the set of all arrival rate vectors $\lambda$ 
where there is a scheduling policy under which the queueing network of the system is stable\footnote{{As mentioned earlier, the} stability condition that we are interested in is rate stability.}. 
The \emph{nominal} traffic rate to all nodes of job type $m$ in the network is $\lambda_m$. Let $\nu = [\nu_k] \in \mathbb{R}^K_+$ be the set of nominal traffic rate of nodes in the network. Then, $\nu_k = \lambda_m$ if $m(k) = m$, i.e., if $\sum_{m'=1}^{m-1} K_{m'}<k \leq \sum_{m'=1}^{m} K_{m'}$. The LP that characterizes the capacity region of the network makes sure that the total service capacity allocated to each node in the network is at least as large as the nominal traffic rate to that node. \yuan{Formally,} the LP -- known as the  \emph{static planning problem} \cite{Harrison2000} -- is defined as follows. 

\begin{eqnarray}
\text{Minimize} & & ~\rho \label{eq:LP} \\ \nonumber
\text{subject to} & & ~\nu_k \leq \sum_{j \in {\cal S}_k} \mu_{kj} p_{kj}, ~ \forall ~ 1 \leq k \leq K\\
 & & ~\rho \geq \sum_{k \in {\cal T}_j} p_{kj}, \quad ~~~ \forall ~ 1 \leq j \leq J, \label{eq:server-allocation} \\
 & & p_{kj} = 0, \quad \quad \quad \quad \text{if $k \not \in \mathcal{T}_j$}, \label{eq:LP3} \\ 
 & & p_{kj} \geq 0.
\end{eqnarray}

\begin{prop}\label{prop:rho}
Let the optimal value of the LP be $\rho^*$. 
Then $\rho^* \leq 1$ is a necessary and sufficient condition of rate stability of the system \ramtin{under some scheduling policy.}
\end{prop}
The proof of Proposition \ref{prop:rho} is provided in Appendix \ref{apdx:prop}.

Thus, by Proposition \ref{prop:rho}, the \emph{capacity region} $\Lambda$ 
of the network is the set of all $\lambda \in \mathbb{R}_+^M$
for which the corresponding optimal solution $\rho^*$ to the LP satisfies $\rho^* \leq 1$. 
More formally, 
\begin{align*}
\Lambda \triangleq \Bigg\{ \lambda \in \mathbb{R}_+^M : \exists ~p_{kj} \geq 0 \text{ such that } 
\sum_{k \in {\cal T}_j} p_{kj} \leq 1 ~~\forall ~j, \text{ and } \nu_k \leq \sum_{j \in {\cal S}_k} \mu_{kj} p_{kj} ~~\forall ~k \Bigg\}.
\end{align*}

\subsection{Scheduling Policy Robust to Task Service Rates}\label{sec:policy}
\yuan{In this subsection,} we make the following assumption on service rates $\mu_{kj}$.
 
\begin{assumption}\label{asmp:factor} 
For all $k$ and all $j \in {\cal S}_k$, service rates $\mu_{kj}$ can be factorized to two terms: a task-dependent term $\mu_k$, and a server-dependent term $\alpha_j$. Thus, $\mu_{kj} = \mu_k \alpha_j$. 
\end{assumption}

While Assumption \ref{asmp:factor} appears somewhat restrictive, it covers 
a variety of important cases. 
When $\alpha_j = 1$ for all $j$, the service rates are task dependent. 
This case models, for example, a data center of servers with the same processing speed 
(possibly of the same generation and purchased from the same company), 
but with different software compatibilities, and possibly hosting overlapping sets of data blocks. 
The case when $\alpha_j$ are different can model the inherent heterogeneous 
processing speeds of the servers. 

We now propose a scheduling policy with known $\alpha_j$, which is robust to task service rates $\mu_k$, 
and prove that it is throughput optimal. 
The idea of our scheduling policy is quite simple: it reacts to queue size changes by adjusting the service allocation vector $p = [p_{kj}]$. 
Since service rates $\mu_{kj}$ are factorized to two terms $\mu_k$ and $\alpha_j$, only the sum $p_k \triangleq  \sum_{j \in {\cal S}_k} \alpha_j p_{kj}$ affects 
the effective service rate for node $k$. One can consider $p_k$ as the total capacity that all the servers allocate to node $k$ in a time slot. So, with a slight abuse of notation and terminology, we call $p = [p_k]$ the service allocation vector from now on.   

To precisely describe our proposed scheduling algorithm, first we introduce some notation. Let $\bOne\{Q^n_{(k',k)} > 0\}$ be the indicator that the queue corresponding to edge $(k',k)$ is non-empty at time $n$.  Let $\Delta Q^{n+1}_{(k',k)} = Q^{n+1}_{(k',k)}-Q^n_{(k',k)}$ be the size change of queue $(k',k)$ from time $n$ to $n+1$. Define $E^n_k$ to be the event that there is a strictly positive number of type-$k$ tasks 
to be processed at time $n$. Thus, $E^n_k = \{Q^n_{(0,k)} > 0\}$ if $k$ is a root node, 
and $E^n_k = \{Q^n_{(k',k)} > 0, ~\forall k' \in {\cal P}_k \}$ if $k$ is not a root node. 
Also let $\bOne_{E^n_k}$ be the indicator function of event $E^n_k$. 

Let $\mathcal{C} \subseteq \mathbb{R}^K_+$ be the polyhedron of feasible service allocation vector $p$.
\begin{align} \label{eq:convex-set}
\mathcal{C} = \Bigg \{p \in \mathbb{R}^K : \exists ~ p_{kj} ~\text{such that} ~ \sum_{j \in {\cal S}_k} \alpha_j p_{kj} = p_k ~~\forall k, 
~ p_{kj}\geq 0 ~~ \forall k,j, ~ \sum_{k \in {\cal T}_j} p_{kj} \leq 1 ~~\forall j \Bigg \}.
\end{align}
For any $K$-dimensional vector $x$, 
let $[x]_{\mathcal{C}}$ denote the convex projection of $x$ onto $\mathcal{C}$. 
Finally, let $\{\beta^n\}$ be a positive decreasing sequence with the following properties: (i) $\sum_{n=1}^{\infty} \beta^n = \infty$, (ii) $\sum_{n=1}^{\infty} (\beta^n)^2 < \infty$, and (iii) $\lim_{n \to \infty} \frac{1}{n \beta^n} < \infty$.
 
As we will see in the sequel, a key step of our algorithm is to find an unbiased estimator of $\lambda - \mu_k p_k$ for all $k$, based on the current and past queue sizes. Toward this end, for each node $k$, 
we first pick a path of queues from a queue corresponding to a root node of the DAG to queue $(k',k)$ for some $k' \in {\cal P}_k$. Note that the choice of this path need not be unique. Let ${\cal H}_k$ denote the set of queues on this path from a root node to node $k$. For example, in the DAG of Figure \ref{fig:DAG}, for node $4$, we can pick the path ${\cal H}_4 = \{ (0,1),(1,2),(2,4)\}$. Then, we use $\sum_{(i',i) \in {\cal H}_k} \Delta Q_{(i',i)}$ as an unbiased estimate of $\lambda - \mu_k p_k$. To illustrate the reason behind this estimate, consider the DAG in Figure \ref{fig:QN}. It is easy to see that 
\begin{align*}
 \mathbb E(\Delta Q^{n+1}_{(0,1)} + \Delta Q^{n+1}_{(1,2)} + \Delta Q^{n+1}_{(2,4)} |Q^n_{(2,4)}>0, Q^n_{(3,4)}>0,p^n)  = \lambda - \mu_4 p^n_4.
\end{align*} 

In general, if $L_k = L-1$ for node $k$ 
(recall that $L_k$ is the length of the longest path from a root node to $k$), one picks a path of edges 
$(i_0,i_1),(i_1,i_2),\ldots,(i_{L-1},i_L),$
such that $i_0 = 0$ and $i_L = k$. Then,
\begin{align} \nonumber
\EE \left[ \sum_{l=0}^{L-1} \Delta Q^{n+1}_{(i_l,i_{l+1})} | \bOne_{E^n_k} = 1,p^n \right] 
&=(\nu_k - \mu_{i_1}p^n_{i_1}\bOne_{E^n_{i_1}}) + \sum_{l=1}^{L-1}  (\mu_{i_l}p^n_{i_l}\bOne_{E^n_{i_l}} - \mu_{i_{l+1}}p^n_{i_{l+1}}\bOne_{E^n_{i_{l+1}}}) \\ \label{eq:unbiased}
&= \nu_k - \mu_k p^n_k \bOne_{E^n_k}.
\end{align}

\ramtin{Note that one can pick any path from a root node to $k$, but the longest path is picked in \eqref{eq:unbiased} for the purpose of ease of notation for the proofs.} Our scheduling algorithm updates the allocation vector $p^n$ in each time slot $n$ 
in the following manner. 
\begin{itemize}
\item[1.] We initialize with an arbitrary feasible $p^0$. 
\item[2.] Update the allocation vector $p^n$ as follows: 
\begin{align}\label{eq:update}
p^{n+1}_k = [p^n_k + \beta^n \bOne_{E^n_k} \sum_{(i',i) \in {\cal H}_k}\Delta Q^{n+1}_{(i',i)}]_{\mathcal{C}}.
\end{align}
\end{itemize}
This completes the description of the algorithm.

We now provide some intuition for the algorithm. As we mentioned, the algorithm tries to find adaptively the capacity allocated to task $k$, $p_k$, that balances the nominal arrival rate and departure rate of queues $(k',k)$. The nominal traffic of all the queues of DAG type $m$ is \ramtin{$\nu_{k}$ for task types $k$ belonging to job type $m$.} Thus, the algorithm tries to find $p^*_k = \frac{\nu_k}{\mu_k}$, in which case the nominal service rate of all the queues is $p^*_k \mu_k = \nu_k$. To find an adaptive robust algorithm, we formulate the following optimization problem. 
\begin{eqnarray}
\text{minimize} & & \frac12\sum_{k=1}^K (\nu_k - \mu_k p_k)^2 \nonumber \\ \label{eq:optimization}
\text{subject to} & & p \in \mathcal{C}.
\end{eqnarray}
Solving \eqref{eq:optimization} by the standard gradient descent algorithm, 
using step size $\beta^n$ at time $n$, leads to the update rule
\begin{equation}\label{eq:update3}
p^{n+1}_k = [p^n_k+\beta^n \mu_k(\nu_k - \mu_k p^n_k)]_{\cal C}.
\end{equation}
To make the update in \eqref{eq:update3} robust, first we consider a ``skewed" update
\begin{equation}\label{eq:skewed-update}
p^{n+1}_k = [p^n_k+\beta^n (\nu_k - \mu_k p^n_k)]_{\cal C},
\end{equation}
and second, we use the queue-length changes in \eqref{eq:unbiased} as an unbiased estimator of the term $\nu_k - \mu_k p^n_k$. This results in the update equation in \eqref{eq:update}. Thus, the update in \eqref{eq:update} becomes robust to knowledge of service rates $\mu_k$. The algorithm is not robust to knowledge of server rates, $\alpha_j$, since the convex set $\cal C$ is dependent on $\alpha_j$, and the projection requires the knowledge of server speeds $\alpha_j$. 

{Let us now provide some remarks on the implementation efficiency of the algorithm. 
First, the policy is not fully distributed. 
While the update variables $\bOne_{E^n_k} \sum_{(i',i) \in {\cal H}_k}\Delta Q^{n+1}_{(i',i)}$ can be 
computed locally, the projection $[\cdot]_{\cal C}$ requires full knowledge of all these local updates. 
Second, since Euclidean projection on a polyhedron is a quadratic programming problem that can be solved efficiently in polynomial time by optimization algorithms such as the ``interior point method" \cite{Boyd}, the projection step $[\cdot]_{\cal C}$ can be implemented efficiently. 
}

The simulation results 
are derived using the described algorithm.  
\yuan{To analyze the theoretical performance properties of the algorithm,} we make \yuan{minimal} modifications to the proposed algorithm \yuan{for technical reasons}. First, we assume that a) the nominal arrival rate of all the tasks $\nu_k$ is strictly positive, b) there are finitely many servers in the system, and c) all the service rates, $\mu_{kj}$, are finite. Note that assumptions a), b), and c) 
\yuan{can be made without loss of generality.} Then, there exists $\veps_0 > 0$ such that for all $k$, $\frac{\nu_k}{\mu_k} \geq \veps_0$. We now suppose that $\veps_0$ is known, 
and consider a variant $\mathcal{C}_{\veps_0}$ of the convex set $\mathcal{C}$, 
defined to be

\begin{align}\label{eq:convex-set2}
\mathcal{C}_{\veps_0} = \Bigg \{p \in \mathbb{R}^K : \exists ~ p_{kj} \geq 0 ~\text{such that} ~ \sum_{j \in {\cal S}_k} \alpha_j p_{kj} = p_k ~~\forall k, ~ p_k\geq \veps_0 ~~ \forall k, ~ \sum_{k \in {\cal T}_j} p_{kj} \leq 1 ~~\forall j \Bigg \}.
\end{align}

Note that $p^* \in \mathcal{C}_{\veps_0}$. We modify the projection to be on the set $\mathcal{C}_{\veps_0}$ every time, so that $p^n$ are now updated as 
\begin{align}\label{eq:update2}
p^{n+1}_k = [p^n_k + \beta^n \bOne_{E^n_k} \sum_{(i',i) \in {\cal H}_k}\Delta Q^{n+1}_{(i',i)}]_{\mathcal{C}_{\veps_0}}.
\end{align}

\yz{
\begin{remark}
a) The modified update \eqref{eq:update2} 
ensures that the service effort $\mu_k p^n_k$ allocated to each queue $k$ 
is always strictly between $0$ and $1$, which will imply that all queues 
are non-empty for a positive fraction of the time, so as to guarantee convergence 
of the modified algorithm \eqref{eq:update2}. 
We believe that the original algorithm \eqref{eq:update} converges as well, 
although establishing this fact rigorously appears difficult.

b) Let us also note that the modified algorithm \eqref{eq:update2} is {\em essentially} robust 
in the following sense. 
On the one hand, the update \eqref{eq:update2} assumes the knowledge of $\veps_0$, 
which in turn depends on $\nu$ and $\mu$. 
On the other hand, $\veps_0$ can be chosen with minimal information on 
$\nu$ and $\mu$. For example, if $c$ is a known lower bound on $\nu_k, k = 1, 2, \cdots, K$ 
and $C$ is a known upper bound on $\mu_k, k = 1, 2, \cdots, K$, 
then we can set $\veps_0 = \frac{c}{C}$.
\end{remark}
}

The main results of this section are the following two theorems. 

\begin{thm}\label{thm1}
Let $\lambda\in \Lambda$.
The allocation vector $p^n$ updated by Equation \eqref{eq:update2} converges to $p^* = [p^*_k]$ almost surely, where $p^*_k = \frac{\nu_k}{\mu_k}$.
\end{thm}
The proof of Theorem \ref{thm1} is provided in Appendix \ref{apdx:proof-thm1}. 
Here, we briefly describe the main steps of the proof. 
First, we show that the non-stochastic gradient projection algorithm with the skewed update \eqref{eq:skewed-update} converges. This is not true in general, but \yuan{the convergence holds here, due to the form of} the objective function in \eqref{eq:optimization}, 
which is the sum of separable quadratic terms. 
Second, we show that the cumulative stochastic noise present in the update due to the error in estimating the correct drift is an $L_2$-bounded martingale. Thus, by martingale convergence theorem the cumulative noise converges and has a vanishing tail. This shows that after some time the noise becomes negligible. Finally, we prove that the event that all queues in the network are non-empty happens 
for a positive fraction of time. Intuitively, this suggests that the algorithm is updating ``often enough'' 
to be able to converge. The rigorous justification makes use of Kronecker's lemma \cite{Durrett}.

\begin{thm}\label{thm2}
Let $\lambda \in \Lambda$.
The queueing network representing the DAGs is rate stable under the proposed scheduling policy, i.e.
$
\lim_{n \to \infty} \frac{Q^n_{(k',k)}}{n} = 0, ~a.s., ~\forall (k',k).
$
\end{thm}
The proof of Theorem \ref{thm2} is provided in Appendix \ref{apdx:proof-thm2}. While proving the theorem is technically quite involved, the key idea is to use Theorem \ref{thm1} to prove that the servers allocate enough cumulative capacity to all the tasks in the \yuan{system}, which leads to rate stability of the network. 


\input{sim}

\section{Flexible Queueing Network}\label{sec:FQN}
In this section, we consider a different queueing network model, and show that our robust scheduling policy can also be applied to this network. 
\subsection{Network Model}
We consider a flexible queueing network with $K$ queues and $J$ servers, 
and probabilistic routing. 
Servers are \emph{flexible} in the sense that each server can serve 
a (non-empty) set of queues. Similarly, 
tasks in each queue are \emph{flexible},
so that each queue can be served by a set of servers. 
Similar to the DAG scheduling model, for each $j$, let $\mathcal{T}_j$ be the set of queues that server $j$ can serve, 
let $T_j = | \mathcal{T}_j |$. 
For each $k$, let $\mathcal{S}_k$ be the set of servers 
that can serve queue $k$, and let $S_k = |\mathcal{S}_k |$. 
Clearly, $\sum_{j=1}^J T_j = \sum_{k=1}^K S_k$, \yuan{and we denote this sum by $S$}. 
Without loss of generality, we assume that each server can serve at least one queue, 
and each queue can be served by at least one server.

We suppose that each queue has a dedicated exogenous arrival process (with rates being possibly zero). 
For each $k$, suppose that arrivals to queue $k$ form an independent Bernoulli process 
with rate $\lambda_k \in [0,1]$. 
Thus, in each time slot, there is exactly one arrival to queue $k$ with probability $\lambda_k$, 
and no arrival with probability $1-\lambda_k$.
Let $A_k(t)$ be the cumulative number of exogenous arrivals to queue $k$ up to time $t$. 
The routing structure of the network is described by the matrix $R = [r_{k'k}]_{1 \leq k',k \leq K}$, 
where $r_{k'k}$ denotes the probability that 
a task from queue $k'$ joins queue $k$ after service completion. 
The random routing is i.i.d. over all time slots. 
We assume that the network is \emph{open}, i.e., 
all tasks eventually leave the system. 
This is characterized by the condition that $(I - R^T)$ is invertible, 
where $I$ is the identity matrix, \yuan{and $R^T$ is the transpose of $R$.}
\begin{example}
To clarify the network model, we consider a flexible queueing network shown in Figure \ref{fig:FQN1}.
For concreteness, we can think of this system as a multi-tier application \cite{eurosys09} with two flexible servers
(the two boxes), and one type of application with three tiers in succession (the three queues). When a task is processed at queue 2, it will join queue 3 with probability $r_{23}$ and it will join queue 1 with probability $r_{21} = 1 - r_{23}$ (that can be thought of as the failure probability in processing queue 2). This network is different from the classical open multiclass queueing networks, in that queue 2 can be served by 2 servers. In this network $\mathcal{T}_1 = \{ 1,2 \}$ and $\mathcal{T}_2 = \{ 2,3 \}$, $\mathcal{S}_1 = \{ 1 \}$, $\mathcal{S}_2 = \{ 1,2 \}$, and $\mathcal{S}_3 = \{ 2 \}$. 
\end{example}

\begin{figure}
\centering
    \includegraphics[width= 0.55\textwidth]{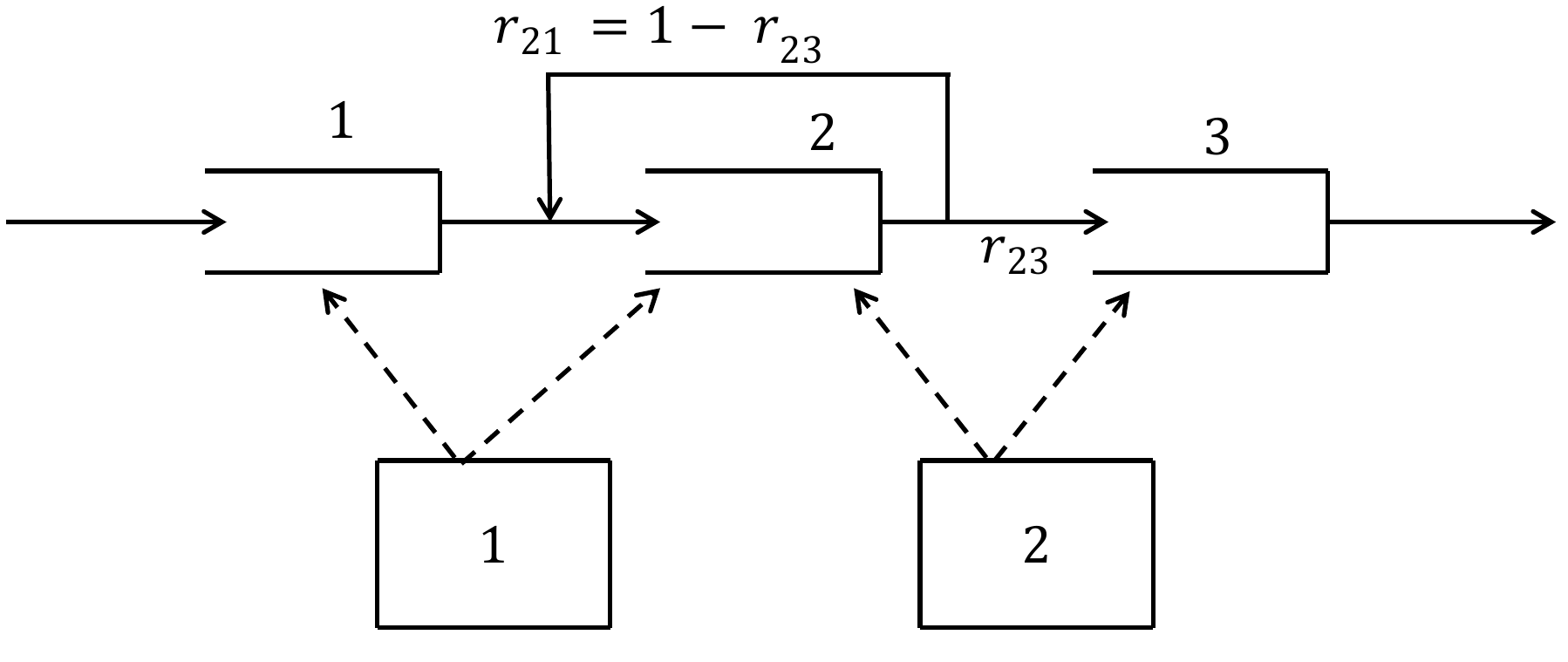}
  \caption{\label{fig:FQN1}Flexible queueing network with 3 queues and 2 servers.}
\end{figure}

\yuan{Similar to the DAG scheduling problem,}
we assume that several servers can work simultaneously on the same task, 
so that their service capacities can be added. 
In each time slot, if a task in queue $k$ is served exclusively by server $j$, 
then the task departs from queue $k$ with probability $\mu_{kj} = \mu_j \alpha_j$, where $\mu_k$ is the service rate of queue $k$ and $\alpha_j$ is the speed of server $j$. 

The dynamics of the flexible queueing network can be \yuan{described} as follows. Let $Q^n_k$ be the length of queue $k$ at time $n$. Let $d^n_k \in \{ 0,1 \}$ be the number of tasks that depart queue $k$ at time $n$. Let $a^n_k \in \{0,1 \}$ be the number of exogenous arrivals to queue $k$ at time $n$. Finally, let $\bOne ^n_{k \to k'}$ be the indicator that the task departing queue $k$ at time $n$ (if any) is destined to queue $k'$. Then the queue dynamics is
\begin{align}\label{eq:FQNdynamic}
Q^{n+1}_k = Q^{n}_k + a^n_k + \sum_{k'=1}^K d^n_{k'} \bOne ^n_{k' \to k}  - d^n_k. 
\end{align}
Note that $\mathbb{E}(a^n_k)= \lambda_k$ and $\mathbb{E}(\bOne ^n_{k' \to k}) = r_{k'k}$.

Similar to the DAG scheduling problem, we define the  allocation vector $p =[p_{kj}]$ (of server capacities), and $p$ is called \emph{feasible} if
\begin{align}\label{eq:feasible2}
\sum_{k \in \mathcal{T}_j} p_{kj} \leq 1, ~ \forall ~ 1 \leq j \leq J.
\end{align}
For each $j$, $p_{kj}$ can
be interpreted as the probability that server $j$
decides to work on queue $k$.
Then, the head-of-the-line task in queue $k$
is served with probability $\sum_j \mu_{kj} p_{kj}$.
Similar to the DAG model, we scale the service rates so that $\sum_j \mu_{kj} p_{kj} < 1$
for any feasible $p$. 

We now introduce the linear program (LP) that characterizes the capacity region of
the flexible queueing network.
Toward this end, for a given arrival rate vector $\lambda$,
we first find the \emph{nominal} traffic rates
$\nu = [\nu_k]_{1\leq k \leq K} \in \mathbb{R}^K$,
where $\nu_k$ is the long-run average total rate at which tasks arrive to queue $k$.
For each $k$, $\nu_k = \lambda_k + \sum_{i=1}^K \nu_i r_{ik}$.
Thus, we can solve $\nu$ in terms of $R$ and $\lambda$:
\begin{align}\label{rate}
\nu = (I-R^T)^{-1} \lambda.
\end{align}
Note that Eq. \eqref{rate} is valid, since by 
\yz{our assumption that the network is open}, $(I - R^T)$ is invertible.
The LP is then defined as follows.
\begin{eqnarray}
\text{Minimize} & & ~~\rho \label{eq:FQNLP} \\
\text{subject to} & & ~\nu_k \leq \sum_{j \in {\cal S}_k} \mu_{kj} p_{kj}, ~ \forall ~ 1 \leq k \leq K, \label{eq:FQNservice-arrival} \\
 & & ~~\rho \geq \sum_{k \in {\cal T}_j} p_{kj}, \quad ~~~ \forall ~ 1 \leq j \leq J, \label{eq:FQNserver-allocation} \\
 & & p_{kj} = 0, \quad \quad \quad \quad \text{if $k \not \in \mathcal{T}_j$}, \label{eq:FQNLP3} \\
 & & p_{kj} \geq 0.
\end{eqnarray}

Let the optimal value of the LP be $\rho^*$.
Similar to Proposition \ref{prop:rho} one can show that $\rho^* \leq 1$
is a necessary and sufficient condition of system stability.
Thus, given $\mu_{kj}$ and $R$, the \emph{capacity region} $\Lambda$
of the network is the set of all $\lambda \in \mathbb{R}_+^K$,
so that the corresponding optimal solution $\rho^*$ to the LP satisfies $\rho^* \leq 1$.

\subsection{Robust Scheduling Policy}\label{sec:FQNpolicy}

In this section, we propose a robust scheduling policy that is provably throughput-optimal when the service rates can be written as $\mu_{kj}=\mu_k \alpha_j$. The policy is robust to arrival and task service rates, but not robust to routing probabilities of the network and servers' speed. The key idea is to use a stochastic gradient projection algorithm to
update the service allocation vector $p$ such that all the flows in the network are balanced. We first give the precise description of the algorithm, and state the main theorem. Then, we provide some explanations. We use similar notation as the one used in Section \ref{sec:DAG}.

Since service rates can be factorized to a task-dependent rate and a server-dependent rate,
only the sum $p_k \triangleq  \sum_{j} \alpha_j p_{kj}$ affects
the effective service rate for queue $k$. So, similar to the DAG scheduling problem, we call $p = [p_k] \in \mathbb{R}^K$ the service allocation vector. 
Our scheduling algorithm updates the allocation vector $p^n$ in each time slot $n$
in the following manner.
\begin{itemize}
\item[1.] We initialize with an arbitrary feasible $p^0$.
\item[2.] Update the allocation vector $p^n$ as follows.
\begin{align}\label{eq:FQNupdate}
p^{n+1} = [p^n + \beta^n \tE^n(I-R^T)^{-1}\Delta Q^n]_{\mathcal{C}_{\veps_0}},
\end{align}
\end{itemize}
where $\tE^n$ is a $K \times K$ diagonal matrix such that $\tE_{kk} = \bOne_{ \{Q^n_k > 0 \} }$ and ${\cal C}_{\veps_0}$ is given in \eqref{eq:convex-set2}.

The main results of this section are the following two theorems.
\begin{thm}\label{thm3}
Let $\lambda\in \Lambda$.
The allocation vector $p^n$ updated by \eqref{eq:FQNupdate} converges to $p^* = [p^*_k]$ almost surely, where $p^*_k = \frac{\nu_k}{\mu_k}$.
\end{thm}
The proof of Theorem \ref{thm3} is almost identical to the proof of Theorem \ref{thm1}. We avoid repeating the details. 
\begin{thm}\label{thm4}
The flexible queueing network is rate stable under the robust scheduling algorithm specified by the update in \eqref{eq:FQNupdate}, i.e.
$$
\lim_{n \to \infty} \frac{Q^n_k}{n} = 0, ~\forall k.
$$
\end{thm}
The proof of Theorem \ref{thm4} is provided in Appendix \ref{apdx:proof-thm4}.

The intuition for the update \eqref{eq:FQNupdate} is as follows. The algorithm tries to adaptively find the allocation vector $p^*$ using a gradient projection method that solves \eqref{eq:optimization}. To robustify the algorithm to the knowledge of task service rates, we consider the ``skewed'' updates in \eqref{eq:skewed-update}. However, the major difference compared to the DAG scheduling problem is the way we find unbiased estimators of the terms $\nu_k - \mu_k p_k^n$. 
We use $\Delta Q^{n+1}$, the changes in queue sizes, and routing matrix $R$,
to estimate these terms.
It is easy to show that the $k^{\text{th}}$ entry of $(I-R^T)^{-1}\Delta Q^{n+1}$ is an unbiased estimator $\nu_k - \mu_k p^n_k$, if $Q^n_k>0$. Define $M = \text{diag}\{\mu_k\}$, 
\yz{the $K \times K$ diagonal matrix with diagonal entries $\mu_k$}. Then, 
\begin{align}
\mathbb E(\tE^n&(I-R^T)^{-1}\Delta Q^{n+1}|Q^n) \\
& = \tE^n (I-R^T)^{-1}\mathbb E(\Delta Q^{n+1}|Q^n) \\
& = \tE^n (I-R^T)^{-1}(\lambda +R^TM \tE^n p^n - M \tE^n p^n) \\
& = \tE^n \nu - M \tE^n p^n \\ \label{eq:FQNunbiased}
& = \tE^n(\nu - M p^n).
\end{align}
Note that matrix $\tE^n$ in update \eqref{eq:FQNupdate} ensures that the algorithm updates $p^n_k$ only for queues $k$ that are non-empty, since $\left[(I-R^T)^{-1}\Delta Q^{n+1}\right]_k$
is no longer an unbiased estimator of $\nu_k- \mu_k p^n_k$ when $Q^n_k = 0$.

\section{Conclusion and Future Work}\label{sec:conclusion}
In this paper, we presented two processing networks that can well model different applications such as cloud computing, manufacturing lines, and healthcare systems. Our processing \yuan{system is} flexible in the sense that \yuan{servers} are capable of processing \yz{different types of} tasks, while tasks can also be served by 
\yz{different} servers. We proposed a scheduling and capacity allocation policy for these networks that is robust to service rates of the tasks and the arrival rates.  The proposed scheduling algorithm is based on solving an optimization problem by stochastic gradient projection. The algorithm solves the problem of balancing all the flows in the network using only queue size information, and uses the allocation vector derived by the gradient algorithm at each time slot as its scheduling decision. We proved rate stability of the queueing networks corresponding to the models in the case that servers are cooperative and service rates can be factorized to a task-dependent rate and server-dependent rate. 

There are many possible directions for future research. We summarize some of these directions as follows. 
\begin{itemize}
\item It is important to find provably throughput optimal and robust scheduling  policies while relaxing the assumption of cooperative servers. \yuan{Some progress has been made in} \cite{PWZAllerton}.
\item A future direction is to investigate whether there exists a throughput-optimal policy which is only dependent on the queue-size information in the network in the current state or in the past, when service rates are generic, and cannot be necessary factorized to a task-dependent rate and server-dependent rate.
\item In the case of flexible queueing networks, a future direction is to find a throughput-optimal policy that is robust to the knowledge of routing probabilities. 
\end{itemize}

\bibliographystyle{apt}
\bibliography{ref-YZ}

\newpage 

\section{Appendix}

\subsection{Proof of Proposition \ref{prop:rho}}\label{apdx:prop}
Consider the fluid scaling of the queueing network, $X^{rt}_{(k',k)} = \frac{Q_{(k',k)}(\lfloor rt \rfloor)}{r}$ (see \cite{Dai95} for more discussion on the stability of fluid models), 
\yz{and let $X^t_{(k', k)}$ be the corresponding fluid limit.} 
The fluid model dynamics is as follows. If $k$ is a root node, then
$$
X^t_{(0,k)} = X^0_{(0,k)} + A^t_{m(k)} - D^t_k,
$$
where $A^t_{m(k)}$ is the total number of jobs of type $m$ (scaled to the fluid level) that have arrived to the system until time $t$. If $k$ is not a root node, then,
$$
X^t_{(k',k)} = X^0_{(k',k)} + D^t_{k'} - D^t_k,
$$
where $D^t_{k}$ is the total number of tasks (scaled to the fluid level) of type $k$ processed up to time $t$. Suppose that $\rho^* > 1$. We show that if $X^0_{(k',k)} = 0$ for all $(k',k)$, there exists $t_0$ and $(k',k)$ such that $X^{t_0}_{(k',k)} \geq \eps(t_0)>0$, which implies that the system is weakly unstable \cite{dai99}. In contrary suppose that there exists a scheduling policy that under that policy for all $t \geq 0$ and all $(k',k)$, $X^t_{(k',k)} = 0$. Pick a regular point\footnote{We define a point $t$ to be regular if $X^t_{(k',k)}$ is differentiable at $t$ for all $(k',k)$.} $t_1$. Then, for all $(k',k)$, $\dot{X}^{t_1}_{(k',k)} = 0$. Since $\dot{A}^{t_1}_{m(k)} = \lambda_m = \nu_k$, this implies that $\dot{D}^{t_1}_{k} = \nu_k$ for all the root nodes $k$. Now considering queues $(k',k)$ such that nodes $k'$ are roots, one gets 
$$\dot{D}^{t_1}_{k} = \dot{D}^{t_1}_{k'} = \nu_{k'} = \nu_k.$$
Similarly, one can inductively show that for all $k$, $\dot{D}^{t_1}_k = \nu_k$. On the other hand, at a regular point $t_1$, $\dot{D}^{t_1}_k$ is exactly the total service capacity allocated to task $k$ at $t_1$. This implies that there exists $p_{kj}$ at time $t_1$ such that $\nu_k = \sum_{j \in {\cal S}_k} \mu_{kj} p_{kj}$ for all $k$ and the allocation vector $[p_{kj}]$ is feasible, i.e. $\sum_{k \in {\cal T}_j} p_{kj} \leq 1$. This contradicts $\rho^* > 1$.

Now suppose that $\rho^* \leq 1$, and $p^* =[p^*_{kj}]$ is an allocation vector that solves the LP. To prove sufficiency of the condition, consider a generalized head-of-the-line processor sharing policy that server $j$ works on task $k$ with capacity $p^*_{kj}$. Then the cumulative service allocated to task $k$ up to time $t$ is $S^t_k = \sum_{j \in {\cal S}_k} \mu_{kj} p^*_{kj} t \geq \nu_k t$. We show that $X^t_{(k',k)} = 0$ for all $t$ and all $(k',k)$, if $X^0_{(k',k)} = 0$ for all $(k',k)$. First consider queue $(0,k)$ corresponding to a root node. Suppose that $X^{t_1}_{(0,k)} \geq \epsilon > 0$ for some positive $t_1$ and $\epsilon$. 
By continuity of \yz{the fluid limit}, there exists \yz{$t_0 \in (0, t_1)$} such that $X^{t_0}_{(0,k)} = \epsilon/2$ and $X^t_{(0,k)} > 0$ for all $t \in [t_0,t_1]$. Then, $\dot{X}^{t}_{(0,k)} = \nu_k - \sum_{j \in {\cal S}_k} \mu_{kj} p^*_{kj} \leq 0$ for $t \in [t_0,t_1]$, which is a contradiction. Now we show that $X^t_{(k',k)} = 0$ for all $t$ if $k'$ is a root node and $k$ is a child of $k'$. Note that $X^t_{(0,k')} = 0$; thus, $\dot{D}^t_{k'} = \nu_{k'} = \nu_k$. Then, $\dot{X}_{(k',k)} = \nu_{k} - \sum_{j \in {\cal S}_{k}} \mu_{kj} p^*_{kj} \leq 0$. This proves that $X^t_{(k',k)} = 0$ for all $t$. One can then inductively complete this proof for all queues $(k',k)$. 

\subsection{Proof of Theorem \ref{thm1}}\label{apdx:proof-thm1}
Recall the following notation which will be widely used in the proofs.
$$
\bOne_{\{Q^n_{(k',k)} > 0, ~\forall k' \in {\cal P}_k\}} = \bOne_{E^n_k}.
$$
We introduce another notation which is
$$
\bOne_{E^n} = \prod_{k=1}^K \bOne_{E^n_k}.
$$
Note that event $E^n$ denotes the event that all the queues are non-empty at time $n$.

\begin{lem}\label{LEM:PFRAC}
\yz{There exist constants $\ell$ and $\delta_0 > 0$, which are independent of $n$, 
such that given any history ${\cal F}^n$ up to time $n$,} 
$\mathbb P(\bOne_{E^{n+\ell}}=1 |{\cal F}^n)\geq \delta_0 > 0$.
\end{lem}

\begin{proof} 
We work with each of the DAGs ${\cal G}_m$ 
separately, and construct events so that all the queues corresponding to ${\cal G}_m$ 
have positive lengths after some time $\ell$. 
We can do this since $\mu_k p^n_k$ will always be no smaller than $\mu_k \veps_0$ and strictly smaller than $1$, so there is positive probability of serving or not serving a task.  

Let $\wE^n_k$ be the event that task $k$ is served at time $n$, $\bE^n_k$ be the event that task $k$ is not served at time $n$, and $\hat{E}^n_m$ be the event that job type $m$ arrives at time $n$. Consider a particular DAG ${\cal G}_m$. Recall that $L_k$ is the length of the longest path from the root nodes of the DAG to node $k$. Let $\ell_m = \max_{k \in {\cal V}_m}  L_k + 1$. We construct the event $E(\ell_m)$ that happens with a strictly positive probability, and assures that all the queues at time $n+ \ell_m$ are non-empty. Toward this end, let $E(\ell_m) = \cap_{n'=0}^{\ell_m - 1} C^{n'}$, where event $C^{n'}$ is
$$
C^{n'} = \hat{E}^{n'}_m \cap_{\{k: L_k \leq n'-1\}}\wE^{n'}_k \cap_{\{k: L_k > n'-1\}} \bE^{n'}_k.
$$
In words, $C^{n'}$ is the event that at time $n'$, there is a job arrival of type $m$, 
services of tasks of class $k$ for $k$ with $L_k \leq n'-1$, and no service 
of tasks of class $k$ for $k$ with $L_k > n'-1$.
Now, by construction all the queues are non-empty at time $n+\ell$ with a positive probability. To illustrate how we construct this event, consider the example of Figure \ref{fig:DAG} and the corresponding queueing network in Figure \ref{fig:QN}. Then, $C^{0}$ is the event that there is an arrival to the system, and no service in the network. $C^{1}$ is the event that there is a new job arriving, task $1$ is served, and tasks $2$, $3$, and $4$ are not served. Note that there is certainly at least one available task $1$ to serve due to $C^0$. Up to now, certainly queues $(0,1)$, $(1,2)$, and $(1,3)$ are non-empty. $C^2$ is the event of having a new arrival, service to tasks $1$, $2$, and $3$, and no service to task $4$. This construction assures that after $3$ time slots, all the queues are non-empty.

Now for the whole network it is sufficient to take $\ell = \max_m \ell_m$. Construct the events $E(\ell_m)$ for each DAG independently, and freeze the DAG ${\cal G}_m$ (no service and no arrivals) from time $n + \ell_m$ to $n + \ell - 1$. This construction makes sure that all the queues in the network are non-empty at time $n+\ell$ given any history ${\cal F}^n$ with some positive probability $\delta_0$. 
\end{proof}

\begin{lem}\label{lem:coupling}
The following inequality holds.
$$
\liminf_{n \to \infty} \frac{1}{n} \sum_{n'=1}^n \bOne_{E^{n'}} \geq \frac{\delta_0}{\ell} > 0, a.s.
$$
\end{lem}
\begin{proof}
Take a subsequence $\bOne_{E^{n'\ell}}, ~n' \geq 1$. 
\yz{Define a sequence $\left(Y^{n'}\right)_{n'\geq 1}$ by
\[
Y^{n'} = \bOne_{E^{n'\ell}} - \PP(\bOne_{E^{n'\ell}}=1 \mid {\cal F}^{(n'-1)\ell}).
\]
Then, it is easy to see that $\EE\left[Y^{n'}\mid {\cal F}^{(n'-1)\ell}\right] = 0$. 
Thus, $\left(Y^{n'}\right)$ is an ${\cal F}^{n'\ell}$-adapted zero-mean martingale. 
Furthermore, we have $|Y^{n'}|\leq 2$ a.s. for each $n'$. By applying the martingale law of large numbers 
(see e.g., Corollary 2 in Section 11.2 of \cite{Chow1997}), we have 
$\lim_{m \rightarrow \infty}\frac{1}{m} \sum_{n'=1}^m Y^{n'} = 0$ a.s. 
By Lemma \ref{LEM:PFRAC}, this immediately implies that
$\liminf_{m \rightarrow \infty} \frac{1}{m} \sum_{n'=1}^m \bOne_{E^{n'\ell}} \geq \delta_0$ a.s. 
Therefore, with probability 1, 
\[
 \liminf_{n \to \infty} \frac{1}{n} \sum_{n'=1}^n \bOne_{E^{n'}} \geq \liminf_{n \to \infty} \frac{1}{n} \sum_{n'=1}^{\lfloor n/\ell \rfloor} \bOne_{E^{n' \ell} } 
 \geq \frac{\delta_0}{\ell}.
\]
This completes the proof of Lemma \ref{lem:coupling}.}
\end{proof}

\begin{lem}\label{lem:kronecker}
The following equality holds. 
$$
\lim_{n \to \infty} \sum_{n'=1}^n \beta^{n'} \bOne_{E^{n'}} = \infty, ~ a.s.
$$
\end{lem}
\begin{proof}
From now on we work with the probability-$1$ event defined in Lemma \ref{lem:coupling}. 
Consider a sample path in this probability-$1$ event, and 
let $x_{n'} = \beta^{n'} \bOne_{ E^{n'} }$. First note that $x_{n'} \geq 0$. Thus, by the monotone convergence theorem, the series either converges or goes to infinity. Suppose that 
$
\lim_{n \to \infty} \sum_{n'=1}^n x_{n'} = s
$
for some finite $s$.
Define the sequence $b_{n'} = \frac{1}{\beta^{n'}}$. Then, by Kronecker's lemma \cite{Durrett}, we have
$
\lim_{n \to \infty} \frac{1}{b_n} \sum_{n'=1}^n b_{n'} x_{n'} = 0.
$
This shows that 
$
\lim_{n \to \infty} \frac{1}{b_n} \sum_{n'=1}^n \bOne_{E^{n'} } = 0,
$
which results in a contradiction, since $\lim_{n \to \infty} \frac{1}{n \beta^n}$ is finite, 
and hence by Lemma \ref{lem:coupling}, 
$$
\liminf_{n \to \infty} \frac{1}{b_n} \sum_{n'=1}^n \bOne_{E^{n'}} > 0.
$$
\end{proof}

Now we are ready to prove Theorem \ref{thm1}. Consider the probability-$1$ event in Lemma \ref{lem:kronecker}. Let $d_n = \| p^n - p^*\|^2$ and fix $\eps > 0$. We prove that there exists a $n_0(\epsilon)$ such that for all $n \geq n_0(\epsilon)$, $d_n$ has the following properties.
\begin{itemize}
\item [(i)] If $d_n < \epsilon$, then $d_{n+1} < 3\epsilon$.
\item [(ii)] If $d_n \geq \epsilon$ then, $d_{n+1} \leq d_n - \gamma^n$ where $\sum_{n=1}^\infty \gamma^n = \infty$ and $\gamma^n \to 0$. 
\end{itemize}

Then property (ii) shows that for some large enough $n_1 = n_1(\eps)> n_0(\eps)$, $d_{n_1} < \eps$, 
and properties (i) and (ii) show that $d_n < 3 \eps$ for $n \geq n_1(\eps)$. This is true for all $\eps > 0$, so $d_n$ converges to $0$ almost surely.

First we show property (i). 
Let $U^n = [U^n_k] \in \mathbb {R}^K$ be the vector of updates such that 
$$U^{n+1}_k = \bOne_{E^n_k} \sum_{(i',i) \in {\cal H}_k}\Delta Q^{n+1}_{(i',i)}.$$
Note that $\| U^n \|^2$ is bounded by some constant $C_1 > 0$, since the queues length changes at each time slot are bounded by $1$. On the other hand, $\beta^n \to 0$. Thus, one can take $n_1(\eps)$ large enough such that for all $n \geq n_1(\eps)$, $\beta^n \leq \sqrt{\frac{\eps}{2C_1}}$. Then, for $n \geq n_1(\eps)$ if $d_n < \eps$,  
\begin{align}
d_{n+1} &= \| p^{n+1} - p^*\|^2 \\
& = \| [p^n + \beta^n U^{n+1}]_{{\cal C}_\veps} - p^* \|^2 \\ \label{eq:non-expansive}
& \leq \| p^n + \beta^n U^{n+1} - p^*\|^2 \\ \label{eq:CS-ineq}
& \leq 2d_n + 2(\beta^n)^2\| U^{n+1} \|^2 < 3\eps,
\end{align}
where \eqref{eq:non-expansive} is due to the fact that projection to the convex set is non-expansive, 
and \eqref{eq:CS-ineq} is by Cauchy-Schwarz inequality.

To show property (ii), we make essential use of the fact that the cumulative stochastic noise is a martingale. Let 
$$Z^{n+1}_k = \bOne_{E^n_k }(U^n_k - \nu_k + \mu_k p^n_k).$$
Then, by \eqref{eq:unbiased}, 
\begin{equation}\label{eq:MG}
\EE \left[Z^{n+1}_k | {\cal F}^n \right] = 0, ~\forall k
\end{equation}
which shows that $Z^n$ is a martingale difference sequence. Now observe that 
\begin{align}
d_{n+1} &= \sum_{k=1}^K (p^{n+1}_k - p^*_k)^2 \\ \label{gamma1}
& \leq \sum_{k=1}^K (p^n_k + \beta^n U^{n+1}_k - p^*_k)^2 \\
&= d_n + (\beta^n)^2 \| U^{n+1} \|^2 + 2 \beta^n \sum_{k=1}^K (p^n_k - p^*_k)(\nu_k - \mu_k p^n_k+Z^{n+1}_k)\bOne_{E^n_k} \\
&= d_n + (\beta^n)^2 \| U^{n+1} \|^2  + 2 \beta^n \sum_{k=1}^K (p^n_k - p^*_k)(\mu_k (p^*_k-p^n_k)+Z^{n+1}_k)\bOne_{E^n_k}
\\ \label{gamma2}
&\leq d_n + (\beta^n)^2 C_1 - \sum_{k=1}^K 2\mu_k\beta^n \bOne_{E^n}(p^n_k-p^*_k)^2 + \sum_{k=1}^K 2\beta^n Z^{n+1}_k(p^n_k - p^*_k)\bOne_{E^n_k}, 
\end{align} 
where \eqref{gamma1} is due to non-expansiveness of projection, and \eqref{gamma2} is due the facts that $E^n \subset E^n_k$ and $\| U^n \|^2 \leq C_1$. Let $\mu^* = \min_k \mu_k$. Since $\sum_{k=1}^K (p^n_k - p^*_k)^2 > \eps$, the following choice of $\gamma^n$ satisfies $d_{n+1} \leq d_n - \gamma^n$:
\begin{align*}
\gamma^n  = -(\beta^n)^2 C_1 + \beta^n \bOne_{E^n} 2\mu^* \eps - \sum_{k=1}^K 2\beta^n Z^{n+1}_k(p^n_k - p^*_k)\bOne_{E^n_k}.
\end{align*}
As $\beta^n \to 0$ as $n \to \infty$, it is easy to see that $\gamma^n \to 0$ almost surely. Thus, to complete the proof of Theorem \ref{thm1}, one needs to show that $\sum_{n=1}^\infty \gamma_n = \infty$ almost surely. Toward this end, note that $\sum_n (\beta^n)^2$ is finite which makes $-\sum_n (\beta^n)^2 C_1$ bounded. By \eqref{eq:MG}, and the facts that $\sum_n (\beta^n)^2 < \infty$ and $\| p^n - p^*\|$ is bounded for all $n$, we get that 
$$
V^n = \sum_{n'=1}^n \sum_{k=1}^K 2\beta^{n'} Z^{n'+1}_k(p^{n'}_k - p^*_k)\bOne_{E^{n'}_k}
$$
is an $L_2$-bounded martingale and by the martingale convergence theorem converges to some bounded random variable almost surely \cite{Durrett}. Finally, 
$$
2\eps\mu^* \sum_{n=1}^\infty \beta^n \bOne_{E^n} = \infty, a.s.
$$
by Lemma \ref{lem:kronecker}. This completes the proof of Theorem \ref{thm1}.

\subsection{Proof of Theorem \ref{thm2}}\label{apdx:proof-thm2}

In this subsection, we provide the proof of Theorem \ref{thm2}. The key idea to prove the rate stability of queues is to first show that the servers allocate enough cumulative capacity to all the tasks in the network. This is formalized in Lemma \ref{LEM:LIMS}. Second, in Lemma \ref{lem:rate}, we show that each queue $(k',k)$ cannot go unstable if task $k$ receives enough service allocation over time, and the traffic rate coming to these queues is nominal. Finally, we use these two conditions to show rate stability of all the queues in the network by mathematical induction.

To prove the theorem, we first introduce some notation. Let $D^n_k$ denote the cumulative number of processed tasks of type $k$ at time $n$. Recall that $d^n_k$ is the number of processed tasks of type $k$ at time $n$. Therefore, $D^n_k=\sum_{n'=1}^n d^{n'}_k$. Let $A^n_m = \sum_{n'=1}^n a^{n'}_m, ~ 1 \leq m \leq M$ be the cumulative number of jobs of type $m$ that have arrived up to time $n$.
Then the queue-length dynamic of queue $(k',k)$ can be written as follows. If $k' \neq 0$, then $Q^n_{(k',k)} = Q^0_{(k',k)} + D^n_{k'} - D^n_k$. If $k' = 0$, then $Q^n_{(k',k)} = Q^0_{(k',k)} + A^n_{m(k)} - D^n_k$.

At time $n$, the probability that one task is served and departed from queue $(k',k)$ is $\mu_k p^n_k$, if all of the queues $(i,k)$ are non-empty for all $i$. We define $s^n_k$ to be a random variable denoting the virtual service that queues $(k',k)$ have received at time $n$, whether there has been an available task $k$ to be processed or not. $s^n_k$ is a Bernoulli random variable with parameter $\mu_k p^n_k$. Then, the cumulative service that queues $(k',k)$ receive up to time $n$ is $S^n_k=\sum_{n'=1}^n s^{n'}_k$ for all $k'$.
Note that the cumulative service is different from the cumulative departure. 
Indeed, $d^n_k=s^n_k \bOne_{E^n_k}$. 

From now on, in the proof of Theorem \ref{thm2}, we consider the probability-$1$ event that $p^n$ converges to $p^*$ stated in Theorem \ref{thm1}.

\begin{lem}\label{LEM:LIMS}
The following equality holds:
\begin{equation}
\lim_{n \to \infty} \frac{S^n_k}{n} = \nu_k, ~a.s., ~\forall k.
\end{equation}
\end{lem} 

\begin{proof} 
By Theorem \ref{thm1}, the sequence $p^n_k$ converges to $\frac{\nu_k}{\mu_k}$ almost surely. Therefore, for all the sample paths in the probability-$1$ event, and for all $\epsilon_1 > 0$, there exists $n_0(\eps_1)$ such that $\|\mu_k p^n_k -  \nu_k \| \leq \epsilon_1$, for all $n > n_0(\eps_1)$. 

Let $\tilde{s}^n_k$ be i.i.d Bernoulli process of parameter $\nu_k - \epsilon_1$. We couple the processes $s^n_k$ and $\tilde{s}^n_k$ as follows. If $s^n_k = 0$, then $\tilde{s}^n_k = 0$. If $s^n_k = 1$, then $\tilde{s}^n_k = 1$ with probability $\frac{\nu_k - \eps_1}{\mu_k p^n_k}$, and $\tilde{s}^n_k = 0$ with probability $1 - \frac{\nu_k - \eps_1}{\mu_k p^n_k}$. Note that $\tilde{s}^n_k$ is still marginally i.i.d Bernoulli process of parameter $\nu_k - \epsilon_1$. Then, 

\begin{align}
\liminf_{n \to \infty} \frac{S^n_k}{n} &\geq \liminf_{n \to \infty} \frac{\sum_{n'=n_0(\eps_1)+1}^n s^{n'}_k}{n} \\ \label{eq:coup1}
&\geq \liminf_{n \to \infty} \frac{\sum_{n'=n_0(\eps_1)+1}^n \tilde{s}^{n'}_k}{n} \\ \label{eq:coup2}
&= \nu_k - \epsilon_1 ~a.s.,
\end{align}
where \eqref{eq:coup1} is by construction of the coupled sequences, and \eqref{eq:coup2} is 
by the strong law of large numbers. 

Let $\bar{s}^n_k$ be i.i.d Bernoulli process of parameter $\nu_k + \epsilon_1$. We couple the processes $s^n_k$ and $\bar{s}^n_k$ as follows. If $s^n_k = 1$, then $\tilde{s}^n_k = 1$. If $s^n_k = 0$, then $\tilde{s}^n_k = 0$ with probability $\frac{1 - (\nu_k + \eps_1)}{1-\mu_k p^n_k}$, and $\tilde{s}^n_k = 1$ with probability $1- \frac{1 - (\nu_k + \eps_1)}{1-\mu_k p^n_k}$. Note that $\bar{s}^n_k$ is still marginally i.i.d Bernoulli process of parameter $\nu_k + \epsilon_1$. Then,

\begin{align}
\limsup_{n \to \infty} \frac{S^n_k}{n} &\leq \limsup_{n \to \infty} \frac{n_0(\eps_1)+\sum_{n'=n_0(\eps_1)+1}^n s^{n'}_k}{n} \\ \label{eq:coup3}
&\leq \limsup_{n \to \infty} \frac{n_0(\eps_1)+ \sum_{n'=n_0(\eps_1)+1}^n \bar{s}^{n'}_k}{n} \\ \label{eq:coup4}
&= \nu_k + \epsilon_1 ~a.s.,
\end{align}
where \eqref{eq:coup3} is by construction of the coupled sequences, and \eqref{eq:coup4} is 
by the strong law of large numbers.
Letting $\epsilon_1 \to 0$, we have
$$
\liminf_{n \to \infty} \frac{S^n_k}{n} =\limsup_{n \to \infty} \frac{S^n_k}{n} = \nu_k, ~ a.s.
$$
\end{proof}

\begin{lem}\label{lem:rate}
Consider a fixed $k$ and all queues $(k',k)$ with $k' \in {\cal P}_k$. Suppose that 
\begin{equation}\label{eq:prop}
\lim_{n \to \infty} \frac{D^n_{k'}}{n} = \nu_k,~ \forall k' \in {\cal P}_k, ~a.s.
\end{equation}
if ${\cal P}_k \neq \emptyset$, and 
\begin{equation}\label{eq:prop2}
\lim_{n \to \infty} \frac{A^n_{m(k)}}{n} = \nu_k, ~a.s.
\end{equation}
if $k$ is a root node. Then, 
$$
\lim_{n \to \infty} \frac{Q^n_{(k',k)}}{n} = 0, a.s.
$$
\end{lem}
\begin{proof}
Before getting to the details of the proof, note that if $k$ is a root node, then we readily know that 
$$\lim_{n \to \infty} \frac{A^n_{m(k)}}{n} = \lambda_m = \nu_k, a.s. $$
Thus, the lemma states that queues $(0,k)$ are rate stable. 
 
We prove the lemma for the general case that node $k$ is not a root node. Similar proof holds for the case of root nodes. First, we show that for all pair of queues $(i,k)$ and $(i',k)$ such that $i,i' \in {\cal P}_k$, we have 
\begin{equation}\label{eq:qd}
\lim_{n \to \infty} \frac{Q^n_{(i,k)} - Q^n_{(i',k)} }{n} = 0, a.s.
\end{equation}

Note that
\begin{align*}
\frac{Q^n_{(i,k)} - Q^n_{(i',k)} }{n}= \frac{D^n_i - D^n_k - (D^n_{i'} - D^n_k)}{n} 
= \frac{D^n_{i'}-D^n_i}{n}.
\end{align*}
Then, by \eqref{eq:prop},
$$
\lim_{n \to \infty} \frac{D^n_i-D^n_{i'}}{n} = \nu_k - \nu_k = 0, a.s.
$$
Second, we show that 
$$
\liminf_{n \to \infty} \frac{Q^n_{(k',k)}}{n} = 0, a.s.
$$
In contrary suppose that in one realization in the probability-$1$ event defined by \eqref{eq:prop} and Lemma \ref{LEM:LIMS},
\begin{equation}\label{eq:liminf}
\liminf_{n \to \infty} \frac{Q^n_{(k',k)}}{n} > 2\eps_2,
\end{equation}
for some $\eps_2 > 0$. This implies that in that realization,
$$
\liminf_{n \to \infty} \frac{Q^0_{(k',k)}+D^n_{k'}-D^n_k}{n} > 2\eps_2.
$$
By \eqref{eq:prop}, the probability that $\lim_{n \to \infty} \frac{Q^0_{(k',k)}+D^n_{k'}}{n} = \nu_k$ is $1$. Thus, in that realization
\begin{equation}\label{eq:liminf2-2}
\limsup_{n \to \infty} \frac{D^n_k}{n} < \nu_k - 2\eps_2.
\end{equation}
On the other hand, \eqref{eq:liminf} shows that there exists $n_0(\eps_2)$ such that for all $n \geq n_0(\eps_2)$, 
\begin{equation}\label{eq:1}
Q^n_{(k',k)} \geq 2n \eps_2.
\end{equation} 
Furthermore, \eqref{eq:qd} shows that there exists $n_1(\eps_2)$ such that for all $i \in {\cal P}_k$ and for all $n\geq n_1(\eps)$, 
\begin{equation}\label{eq:2}
|Q^n_{(k',k)}-Q^n_{(i,k)}| < n\eps_2.
\end{equation} 
Let $n_2(\eps_2) = \max(n_0(\eps_2),n_1(\eps_2))$. \eqref{eq:1} and \eqref{eq:2} imply that for all $n\geq n_2(\eps_2)$, $Q^n_{(i,k)} \geq n \eps_2$. Now taking $n_3(\eps_2) = \max(n_2(\eps_2),1/\eps_2)$, we have that all the queues $(i,k), ~i \in {\cal P}_k$ are non-empty for $n \geq n_3(\eps_2)$. Thus, $s^n_k = d^n_k$ for all $n \geq n_3(\eps_2)$. 
Therefore,
\begin{align*}
\limsup_{n \to \infty} \frac{S^n_k}{n} \leq \limsup_{n \to \infty} \frac{n_3(\eps_2) + D^n_k}{n}  = \limsup_{n \to \infty} \frac{D^n_k}{n}.
\end{align*}
Thus, by Lemma \ref{LEM:LIMS}, $\nu_k \leq \limsup_{n \to \infty} \frac{D^n_k}{n}$ which contradicts \eqref{eq:liminf2-2}. Since this holds for any $\epsilon_2 >0$, we conclude that
\begin{equation}\label{eq:liminf3-2} 
\liminf_{n \to \infty} \frac{Q^n_{(k',k)}}{n} = 0, ~ a.s.,
\end{equation}
for all $k' \in {\cal P}_k$. 

Third, we show that 
$$
\limsup_{n \to \infty} \frac{Q^n_{(k',k)}}{n} = 0, a.s.
$$

In contrary suppose that in one realization in the probability-$1$ event defined by \eqref{eq:prop}, \eqref{eq:qd}, and Lemma \ref{LEM:LIMS},
\begin{equation}\label{eq:limsup}
\limsup_{n \to \infty} \frac{Q^n_{(k',k)}}{n} > 4 \eps_3,
\end{equation}
for some $\eps_3 > 0$. This implies that in that realization $Q^n_{(k',k)} > 4 n \eps_3$ happens infinitely often. Moreover, by \eqref{eq:liminf3-2}, $Q^n_{(k',k)} < 2 n \eps_3$ happens also infinitely often in that realization. On the other hand, by \eqref{eq:qd}, there exists some $n_0(\eps_3)$ such that for all $n \geq n_0(\eps_3)$ and all $i \in {\cal P}_k$, 
\begin{equation}\label{eq:qd2}
|Q^n_{(i,k)}-Q^n_{(k',k)}| < n \eps_3.
\end{equation} 

Take $N_1 = \max(n_0(\eps_3), \frac{2}{\eps_3})$. Then, there exists $N_1 \leq n_1(\eps_3) < n_2(\eps_3)$ such that $Q^{n_1}_{(k',k)} \leq 2 n_1 \eps_3$ and $Q^{n_2}_{(k',k)} \geq 4 n_2 \eps_3$. In words, $n_1 + 1$ is the first time after $N_1$ that $\frac{Q^{n}_{(k',k)}}{n}$ crosses $2 \eps_3$ without going below $2 \eps_3$ before exceeding $4 \eps_3$. Then, since the queue-length changes by at most $1$ each time slot, queue $(k',k)$ is non-empty for all $n, ~n_1 \leq n \leq n_2$. Furthermore, for all $n$, $n_1 \leq n \leq n_2$ and for all $i \in {\cal P}_k$, by \eqref{eq:qd2},
\begin{align*}
Q^n_{(i,k)}  > Q^n_{(k',k)} - n \eps_3 \geq 2n \eps_3 -1  - n\eps_3  \geq n_1 \eps_3 -1 \geq 1.
\end{align*} 
Thus, all the queues $(i,k)$ are also non-empty for all $n$ in the interval $n_1 \leq n \leq n_2$. Consequently, for all $n, ~ n_1 \leq n \leq n_2$, $s^n_k = d^n_k$. Now define a process 
$$B^n_{(k',k)} = D^n_{k'} - S^n_k.$$
Note that by \eqref{eq:prop} and Lemma \ref{LEM:LIMS}, in the realization of probability-$1$ event that we consider, 
\begin{equation}\label{eq:b1}
\lim_{n \to \infty} \frac{B^n_{(k',k)}}{n} = \nu_k - \nu_k = 0.
\end{equation}
We bound $B^{n_2}_{(k',k)}$ as follows. 
\begin{align*}
B^{n_2}_{(k',k)} &= B^{n_1}_{(k',k)} + [B^{n_2}_{(k',k)}-B^{n_1}_{(k',k)}] \\
&= B^{n_1}_{(k',k)} + [D^{n_2}_{k'} - D^{n_1}_{k'} - (S^{n_2}_k - S^{n_1}_k)] \\
&= B^{n_1}_{(k',k)} + [D^{n_2}_{k'} - D^{n_1}_{k'} - (D^{n_2}_k - D^{n_1}_k)] \\
&= B^{n_1}_{(k',k)} + [Q^{n_2}_k - Q^{n_1}_k] \\
& \geq B^{n_1}_{(k',k)} + 4 \eps_3 n_2 - 2 \eps_3 n_1.
\end{align*}
Dividing both sides of the inequality by $n_2$ and subtracting $B^{n_1}_{(k',k)}/n_1$, one gets
$$
\frac{B^{n_2}_{(k',k)}}{n_2} - \frac{B^{n_1}_{(k',k)}}{n_1} \geq \frac{B^{n_1}_{(k',k)}}{n_1}(\frac{n_1}{n_2}-1) + 4\eps_3 - 2\eps_3 \frac{n_1}{n_2}.
$$
By \eqref{eq:b1}, one can choose a large enough $N_2$ such that for all $n \geq N_2$, 
$|\frac{B^n_{(k',k)}}{n}| \leq \frac{2\eps_3}{3}$. Then, $N_1$ can be chosen as 
$$
N_1 = \max(n_0(\eps_3), \frac{2}{\eps_3}, N_2),
$$
and one chooses $n_1$ and $n_2$ accordingly as before. Then, since $n_1, n_2 \geq N_1$, one can write 
\begin{equation}\label{eq:b2}
|\frac{B^n_{(k',k)}}{n} - \frac{B^n_{(k',k)}}{n}| \leq \frac{4\eps_3}{3}.
\end{equation}
However, 
$$
\frac{B^{n_2}_{(k',k)}}{n_2} - \frac{B^{n_1}_{(k',k)}}{n_1} \geq 2 \eps_3(\frac{n_1}{n_2}-1) + 4\eps_3 - 2\eps_3 \frac{n_1}{n_2} = 2\eps_3,
$$
which contradicts \eqref{eq:b2}. Thus, 
$$
\limsup_{n \to \infty} \frac{Q^n_{(k',k)}}{n} = 0, a.s.
$$
The result holds for arbitrary $k' \in {\cal P}_k$. This completes the proof of Lemma \ref{lem:rate}. 
\end{proof}

Now we are ready to prove Theorem \ref{thm2}. We complete the proof of Theorem \ref{thm2} by induction. Recall that $L_k$ is the length of the longest path from the root of the DAG ${\cal G}_{m(k)}$ to  node $k$. If $k$ is a root, $L_k = 0$. The formal induction goes as follows. 
\begin{itemize}
\item \textbf{Basis:} All the queues corresponding to root nodes, i.e. all $(k',k)$ for which $L_k = 0$ are rate stable. 
\item \textbf{Inductive Step:} If all the queues $(k',k)$ for which $L_k \leq L-1$ are rate stable, then all the queues $(k',k)$ for which $L_k = L$ are also rate stable.
\end{itemize}
The basis is true by Lemma \ref{lem:rate}. The inductive step is also easy to show using Lemma \ref{lem:rate}. For a particular queue $(k',k) = (i_L,i_{L+1})$, suppose that $L_k = L$. Pick a path of edges 
$$ 
(i_1,i_2),(i_2,i_3),\ldots,(i_L,i_{L+1}),
$$
from queue $(0,i_1)$ to $(k',k)$. By assumption of induction, all the queues $(i_l,i_{l+1})$ are rate stable for $l \leq L-1$. 
\begin{align*}
A^n_{m(k)} - D^n_{i_L} & =  A^n_{m(k)} - D^n_{i_1} + \sum_{l=1}^{L-1} (D^n_{i_l}-D^n_{i_{l+1}}) \\
& = Q^n_{(0,i_1)} - Q^0_{(0,i_1)} + \sum_{l=1}^{L-1} (Q^n_{(i_l,i_{l+1})}-Q^0_{(i_l,i_{l+1})}).
\end{align*}
Therefore,
\begin{align*}
\lim_{n\to \infty} \frac{D^n_{i_L}}{n} &= \lim_{n\to \infty} \Big[ \frac{A^n_{m(k)}}{n} -\frac{Q^n_{(0,i_1)} - Q^0_{(0,i_1)}}{n} - \sum_{l=1}^{L-1} \frac{(Q^n_{(i_l,i_{l+1})}-Q^0_{(i_l,i_{l+1})})}{n} \Big] \\
& = \lambda_m = \nu_{{k}}, ~a.s.
\end{align*}
Now since $\lim_{n \to \infty} \frac{D^n_{k'}}{n}=\nu_{{k}} ~a.s.$, by Lemma \ref{lem:rate}, $(k',k)$ is rate stable. This completes the proof of the induction step and as a result the proof of Theorem \ref{thm2}.

\subsection{Proof of Theorem \ref{thm4}}\label{apdx:proof-thm4}
Let $D^n_k$ denote the cumulative number of tasks that have departed queue $k$ by and including time $n$: $D^n_k = \sum_{n'=1}^n d^{n'}_k$. Define $s^n_k$ to be a random variable denoting the virtual service that queue $k$ receives at time $n$, whether the queue has been empty or not. $s^n_k$ is a Bernoulli random variable with parameter $\mu_k p^n_k$. Note that $d^n_k = s^n_k \bOne_{Q^n_k > 0}$. Define the cumulative service that queue $k$ has received up to time $n$ to be $S^n_k = \sum_{n'=1}^n s^n_k$.  

\begin{lem}\label{lem:FQN}
The following equality holds:
\begin{align}
\lim_{n \to \infty} \frac{\sum_{n'=1}^n s^{n'}_k \bOne^{n'}_{k \to k'}}{n} = \nu_k r_{kk'}, ~a.s., ~\forall k,k'.
\end{align}
\end{lem}
\begin{proof}
First note that the sequence of random variables $\bOne^{n'}_{k \to k'}$ is i.i.d. Bernoulli-distributed with parameter $r_{kk'}$, and independent of the sequence $s^{n'}_k$. Now by Theorem \ref{thm3}, the sequence $\mu_k p^{n'}_k$ converges to $\nu_k$ almost surely. Thus, in this probability-1 event, for all $\epsilon_4 > 0$, there exists $n_0(\epsilon_4)$ such that $\| \mu_k p^{n'}_k - \nu_k \| \leq \epsilon_4$ for all $n' > n_0(\epsilon_4)$. 

Let $w^n_k$ be an i.i.d Bernoulli process of parameter $(\nu_k - \epsilon_4)r_{kk'}$. We couple the processes $s^n_k \bOne^{n}_{k \to k'}$ and $w^n_k$ as follows. If $s^n_k = 0$, then $w^n_k = 0$. If $s^n_k = 1$, then $w^n_k = \bOne^{n}_{k \to k'}$ with probability $\frac{\nu_k - \eps_4}{\mu_k p^n_k}$, and $w^n_k = 0$ with probability $1 - \frac{\nu_k - \eps_4}{\mu_k p^n_k}$. $w^n_k$ is still marginally i.i.d. Bernoulli process of parameter $(\nu_k - \epsilon_4)r_{kk'}$. Then, 
\begin{align*}
\liminf_{n \to \infty} \frac{\sum_{n'=1}^n s^{n'}_k \bOne^{n'}_{k \to k'}}{n}  \geq \liminf_{n \to \infty} \frac{\sum_{n'=n_0(\epsilon_4)+1}^n w^{n'}_k}{n} = (\nu_k - \epsilon_4)r_{kk'} ~a.s.
\end{align*}
Now we couple the processes $s^n_k \bOne^{n}_{k \to k'}$ and $v^n_k$, where $v^n_k$ is an i.i.d Bernoulli process of parameter $(\nu_k - \epsilon_4)r_{kk'}$. If $s^n_k = 1$, then $v^n_k = \bOne^{n}_{k \to k'}$. If $s^n_k = 0$, then $v^n_k = 0$ with probability $\frac{1 - (\nu_k + \eps_4)}{1 - \mu_k p^n_k}$, and $v^n_k = \bOne^{n}_{k \to k'}$ with probability $1 - \frac{1 - (\nu_k + \eps_4)}{1 - \mu_k p^n_k}$. $v^n_k$ is still marginally i.i.d. Bernoulli process of parameter $(\nu_k + \epsilon_4)r_{kk'}$. Then, 
\begin{align*}
\limsup_{n \to \infty} \frac{\sum_{n'=1}^n s^{n'}_k \bOne^{n'}_{k \to k'}}{n} \leq \limsup_{n \to \infty} \frac{\sum_{n'=n_0(\epsilon_4)+1}^n v^{n'}_k}{n} = (\nu_k + \epsilon_4)r_{kk'} ~a.s.
\end{align*}
The proof is complete by letting $\epsilon_4 \to 0$.
\end{proof}
Now we are ready to complete the proof of the theorem. Observe that 
\begin{align*}
Q^n_k &= Q^0_k + A^n_k + \sum_{n'=1}^n \sum_{k'=1}^K d^{n'}_{k'} \bOne^{n'}_{k' \to k} - D^n_k \\
& \leq Q^0_k + A^n_k + \sum_{n'=1}^n \sum_{k'=1}^K s^{n'}_{k'} \bOne^{n'}_{k' \to k}- D^n_k.
\end{align*}
So it is enough to show that 
$$
\lim_{n \to \infty} \frac{A^n_k + \sum_{n'=1}^n \sum_{k'=1}^K s^{n'}_{k'} \bOne^{n'}_{k' \to k}- D^n_k}{n} = 0.
$$
First, we show that 
$$
\liminf_{n \to \infty} \frac{A^n_k + \sum_{n'=1}^n \sum_{k'=1}^K s^{n'}_{k'}\bOne^{n'}_{k' \to k} - D^n_k}{n}=0, ~a.s.
$$ 
In contrary suppose that in a realization, 
$$
\liminf_{n \to \infty} \frac{A^n_k + \sum_{n'=1}^n \sum_{k'=1}^K s^{n'}_{k'}\bOne^{n'}_{k' \to k} - D^n_k}{n}>\epsilon_5,
$$ 
for some $\epsilon_5 >0$. Then, using Lemma \ref{lem:FQN} and the fact that $\lim_{n\to \infty} \frac{A^n_k}{n} = \lambda_k$, we have
\begin{align}\label{eq:FQNliminf}
\limsup_{n \to \infty} \frac{D^n_k}{n} < \lambda_k + \sum_{k'=1}^K \nu_{k'} r_{k'k} - \epsilon_5 = \nu_k - \epsilon_5.
\end{align}
On the other hand, $\liminf_{n \to \infty} \frac{Q^n_k}{n}> \epsilon_5$ implies that there exists $n_0(\epsilon_5) > \frac{1}{\epsilon_5}$ such that for all $n>n_0(\epsilon_5)$, $Q^n_k \geq \epsilon_5 n$, or in words, the queue is non-empty after $n_0(\epsilon_5)$. Thus, $s^{n'}_k = d^{n'}_k$ for $n' > n_0$. Therefore,
\begin{align*}
\limsup_{n \to \infty} \frac{S^n_k}{n}  \leq \limsup_{n \to \infty} \frac{n_0 + D^n_k}{n}  = \limsup_{n \to \infty} \frac{D^n_k}{n}.
\end{align*}
Now by Lemma \ref{LEM:LIMS}, $\limsup_{n \to \infty} \frac{S^n_k}{n} = \nu_k \leq \limsup_{n \to \infty} \frac{D^n_k}{n}$ which contradicts \eqref{eq:FQNliminf}.\footnote{The lemma is also valid for the flexible queueing network, and the proof does not change.} Since this holds for any $\epsilon_5 >0$, we conclude that
\begin{align}\label{eq:FQNliminf1} 
\liminf_{n \to \infty} \frac{Q^n_k}{n} = 0, ~ a.s.
\end{align}
Second, we show that 
$$
\limsup_{n \to \infty}  \frac{A^n_k + \sum_{n'=1}^n \sum_{k'=1}^K s^{n'}_{k'}\bOne^{n'}_{k' \to k} - D^n_k}{n}=0, a.s.
$$
Suppose that in a realization 
$$
\limsup_{n \to \infty} \frac{A^n_k + \sum_{n'=1}^n \sum_{k'=1}^K s^{n'}_{k'}\bOne^{n'}_{k' \to k} - D^n_k}{n}> 2\epsilon_6,
$$ 
for some $\epsilon_6 > 0$. This implies that in this realization, $Q^n_k > 2\epsilon_6 n$ happens infinitely often and
$$
\limsup_{n \to \infty} \frac{A^n_k + \sum_{n'=1}^n \sum_{k'=1}^K s^{n'}_{k'}\bOne^{n'}_{k' \to k} - D^n_k}{n} > 2\epsilon_6
$$
in that realization.
Moreover, by \eqref{eq:FQNliminf1}, for any $\epsilon_6 >0$, $Q^n_k < \epsilon_6 n$ happens infinitely often with probability 1. Let $N_2 \geq \frac{2}{\epsilon_6}$. Then, there exist $N_2 \leq n_3 < n_4$ such that $Q^{n_3}_k \leq \epsilon_6 n_3$ and $Q^{n_4}_k \geq 2 \epsilon_6 n_4$ and queue $k$ is non-empty between times $n_3$ and $n_4$. Define a process
$$
\tB^n_k = A^n_k + \sum_{n'=1}^n \sum_{k'=1}^K s^{n'}_{k'}\bOne^{n'}_{k' \to k} - S^n_k.
$$  
Then,
\begin{align}
\tB^{n_4}_k &= \tB^{n_3}_k + [\tB^{n_4}_k - \tB^{n_3}_k] \\ \label{eq:B1}
& \geq \tB^{n_3}_k + Q^{n_4}_k - Q^{n_3}_k \\ \label{eq:B3}
& \geq \tB^{n_3}_k + 2 \epsilon_6 n_4 - \epsilon_6 n_3.
\end{align}
\eqref{eq:B1} is due to the following. 
\begin{align}
&\tB^{n_4}_k - \tB^{n_3}_k \nonumber \\
& =  A^{n_4}_k - A^{n_3}_k 
+ \sum_{n'=n_3+1}^{n_4} \sum_{k'=1}^K s^{n'}_{k'}\bOne^{n'}_{k' \to k} - (S^{n_4}_k - S^{n_3}_k) \nonumber \\
& \geq  A^{n_4}_k - A^{n_3}_k 
+ \sum_{n'=n_3+1}^{n_4} \sum_{k'=1}^K d^{n'}_{k'} \bOne^{n'}_{k' \to k} - (S^{n_4}_k - S^{n_3}_k) \nonumber \\ \label{eq:B2}
& =  A^{n_4}_k - A^{n_3}_k 
+ \sum_{n'=n_3+1}^{n_4} \sum_{k'=1}^K d^{n'}_{k'} \bOne^{n'}_{k' \to k} - (D^{n_4}_k - D^{n_3}_k) \\
& =  Q^{n_4}_k - Q^{n_3}_k. \nonumber
\end{align}
\eqref{eq:B2} is true since queue $k$ is non-empty between times $n_3$ 
and $n_4$. Now \eqref{eq:B3} implies that 
$$
\frac{\tB^{n_4}_k}{n_4}-\frac{\tB^{n_3}_k}{n_3} \geq \frac{\tB^{n_3}_k}{n_3}(\frac{n_3}{n_4}-1)+2 \epsilon_6 - \epsilon_6 \frac{n_3}{n_4}.
$$
By Lemma \ref{LEM:LIMS} we know that
$$
\lim_{n \to \infty} \frac{\tB^n_k}{n} = 0 ~ a.s., 
$$
so one can choose $N_2$ large enough such that for all $n \geq N_2$, $|\tB^n_k/n | \leq \epsilon_6/3$. Then,
\begin{align}\label{eq:B4}
|\frac{\tB^{n_4}_k}{n_4}-\frac{\tB^{n_3}_k}{n_3} | \leq \frac{2 \epsilon_6}{3}.
\end{align}  
However, since $\frac{\tB^{n_3}_k}{n_3} \leq \epsilon_6$ and $\frac{n_3}{n_4} < 1$,
$$
\frac{\tB^{n_4}_k}{n_4}-\frac{\tB^{n_3}_k}{n_3} \geq \epsilon_6(\frac{n_3}{n_4}-1)+2 \epsilon_6 - \epsilon_6 \frac{n_3}{n_4} = \epsilon_6,
$$
which contradicts \eqref{eq:B4}. Thus, 
$$
\limsup_{n \to \infty}  \frac{A^n_k + \sum_{n'=1}^n \sum_{k'=1}^K s^{n'}_{k'}\bOne^{n'}_{k' \to k} - D^n_k}{n}=0, a.s.,
$$
which completes the proof of Theorem \ref{thm4}.

\end{document}

%% file: sim.tex
\subsection{Simulations}\label{sec:sim}
In this section, we show the simulation results and discuss the performance of the robust scheduling algorithm. Consider the DAG shown in Figure \ref{fig:DAG5}. We assume that the system has 
1 type of jobs with arrival rate $\lambda = 0.23$. The task service rates are
$$\mu_1 = 1, \mu_2 = 4/3,  \mu_3 = 2, \mu_4 = 1/2 \text{ and } \mu_5 = 2/3.$$
The server speeds are
$
\alpha_1 = 1 \text{ and } \alpha_2 = 1/2,
$
and ${\cal T}_1 = \{1, 4, 5\}$ and ${\cal T}_2 = \{2, 3, 4\}$.
The step size of the algorithm is chosen to be $\beta^n = \frac{1}{n^{0.6}}$ and the initial queue lengths are $[0,0,0]$. From \eqref{eq:LP}, one can compute that the capacity region $\Lambda$ is $\{\lambda\geq 0: \lambda \leq \frac{6}{23}\}$.  

\begin{figure}
\centering
    \includegraphics[width= 0.35\textwidth]{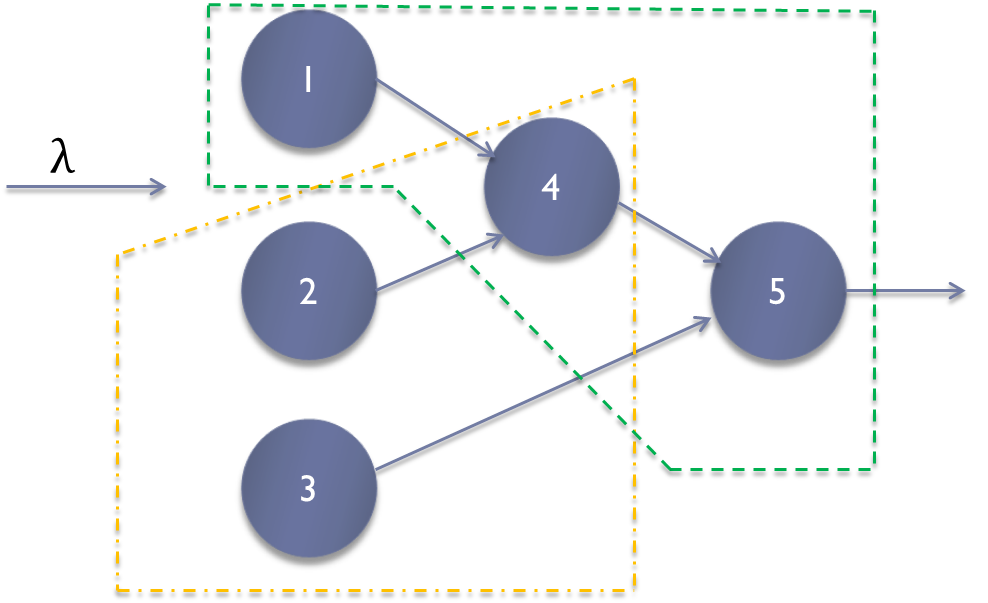}
  \caption{\label{fig:DAG5}DAG of 5 tasks} 
\end{figure}

First we demonstrate that our proposed algorithm is throughput optimal and makes the queues stable. Figure \ref{fig:sim1} illustrates the queue-lengths as a function of time. Moreover, Figure \ref{fig:sim2} shows how vector $p^n$ converges to the flow-balancing values as Theorem \ref{thm1} states.  Figure \ref{fig:sim1} suggests that queues in the network become empty infinitely often, hence are stable. However, the average queue-length is quite large, so the algorithm suffers from bad delay. The reason is that the allocation vector $p^n$ is converging to the value that equalizes the arrival and service rates of all the queues. As an example, if we consider 
a system with a single-node DAG of arrival rate $\lambda$ and a single server of service rate $\mu$, the queueing network reduces to the classical M/M/1 queue. Theorem \ref{thm1} shows that the service capacity that this queue receives, $p^n$, converges to $\frac{\lambda}{\mu}$. It is known that if the arrival rate of an M/M/1 queue is equal to its service rate, the underlying Markov chain describing the queue-length evolution is null-recurrent, and the queue suffers from large delay. 
In the following, we propose a modified version of the algorithm that reduces the delays. \\ \\
\begin{figure}
        \centering
        \begin{subfigure}{0.45\textwidth}
 			\includegraphics[width= 4.8cm, height = 3.6cm]{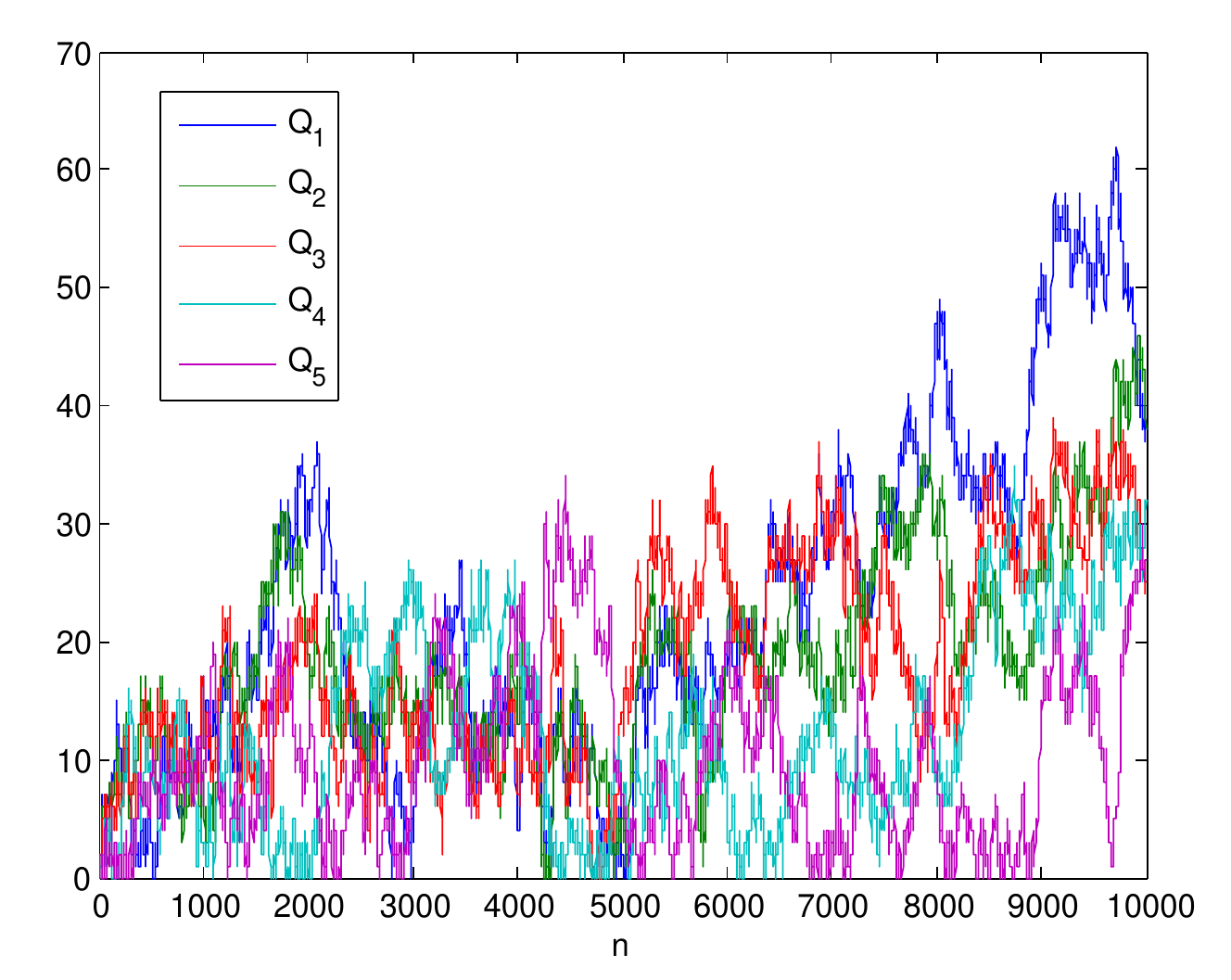}
		 	\caption{\label{fig:sim1}Queue-lengths vs. time} 
        \end{subfigure}
        \qquad
        \begin{subfigure}{0.45\textwidth}
         \includegraphics[width= 4.8cm, height = 3.6cm]{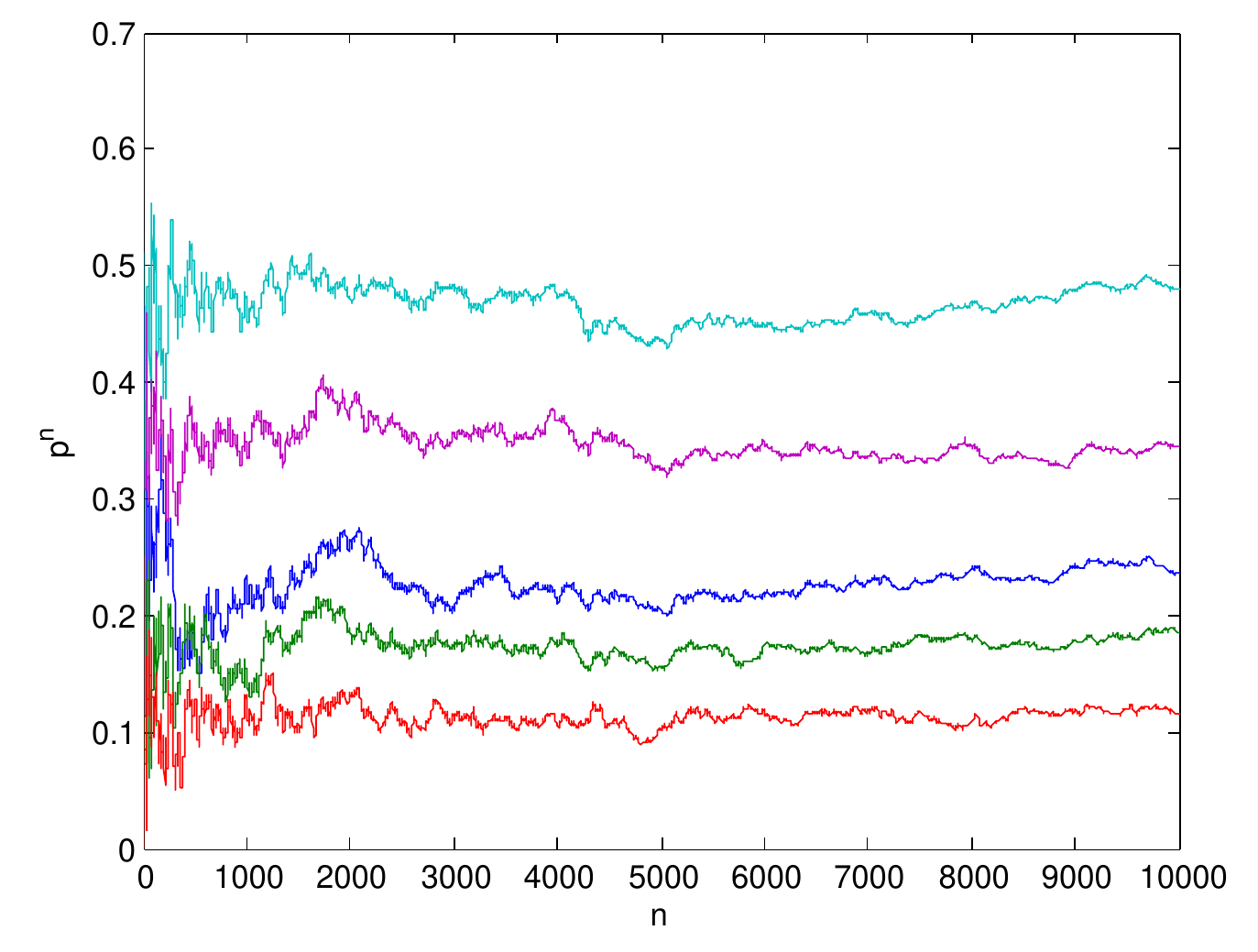}
			 \caption{\label{fig:sim2}Allocation vector $p^n$ vs. time}
        \end{subfigure}
        \vskip\baselineskip
        \begin{subfigure}{0.45\textwidth}
 			\includegraphics[width= 4.8cm, height = 3.6cm]{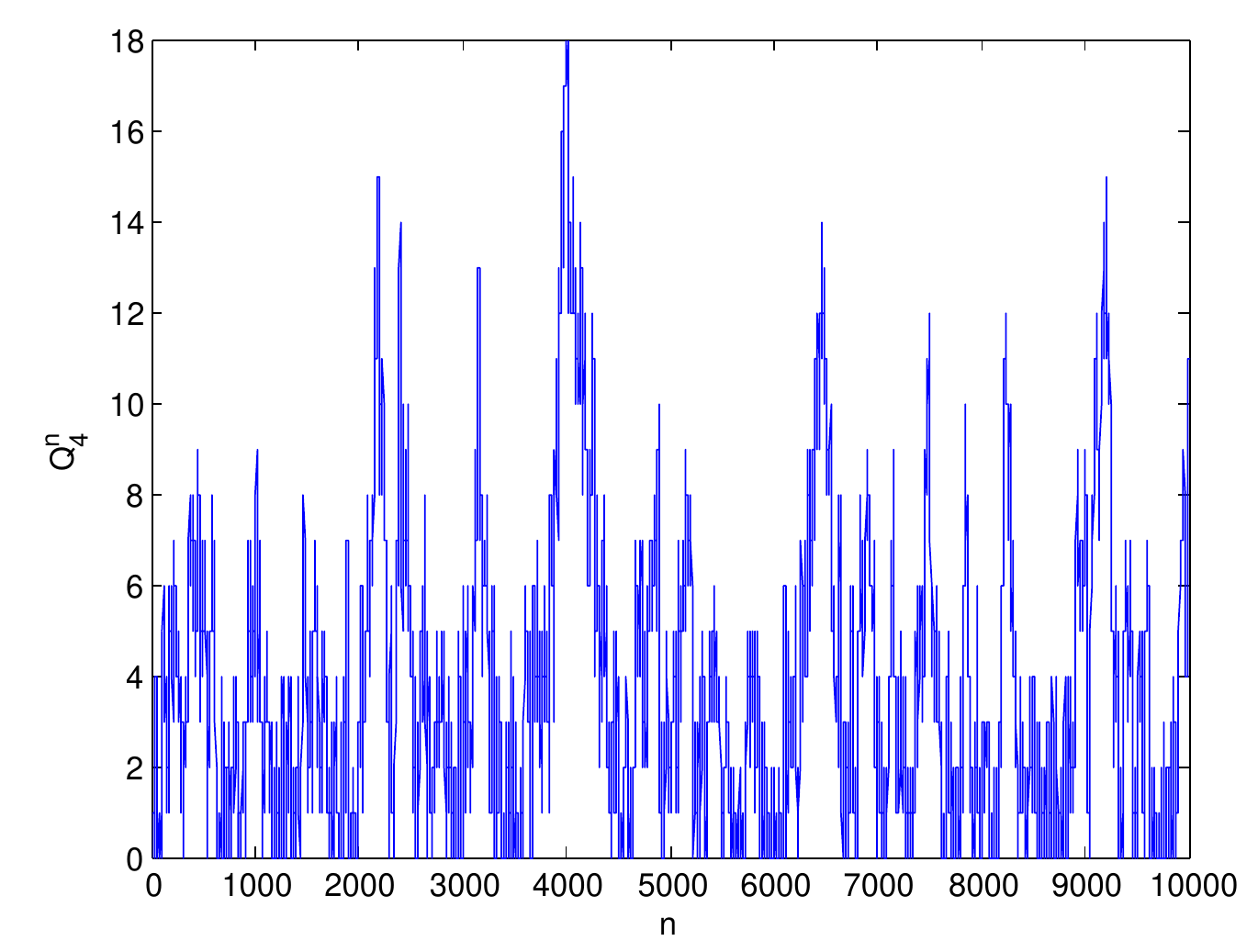}
  \caption{\label{fig:sim3}Queue-length of queue 4 for the modified algorithm vs. time for $\delta = 0.02$.} 
        \end{subfigure}
        \qquad
        \begin{subfigure}{0.45\textwidth}
			\includegraphics[width= 4.8cm, height = 3.6cm]{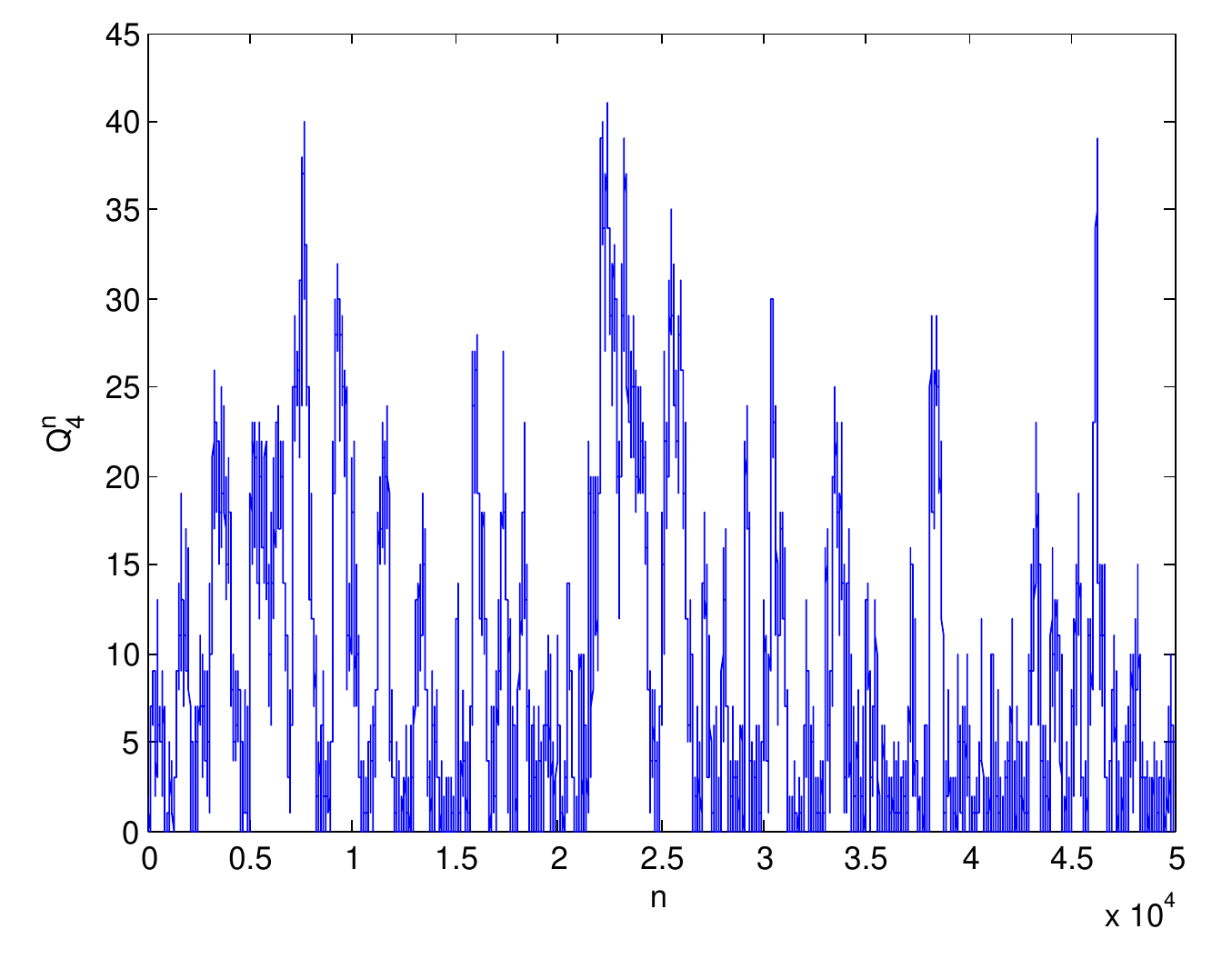}
  			\caption{\label{fig:sim4}Queue-length evolution of queue 4 in the time-varying bursty case.} 
        \end{subfigure}
        
        \caption{This figure shows the simulation results for the DAG of Figure \ref{fig:DAG5}.}\label{fig:DAG-sim}
\end{figure}
{\bf Modified scheduling algorithm to improve delays.}
As discussed in the previous section, the allocation vector $p^n$ converges to the value that just equalizes the arrival rate and the effective service rate that each tasks receives. To improve the delay of the system, one wants to allocate strictly larger service rate to each task than the arrival rate. This is possible only if the arrival vector $\lambda$ is in the interior of the capacity region. In this case, there exists some $\delta > 0$ and an allocation vector $p^*$ such that $\nu_k \leq -\delta + \mu_k p^*_k$ for all $k$. 

Thus, assuming that $\delta$ is known, we minimize the function 
$$
V(p) = \sum_{k=1}^K (\nu_k + \delta - \mu_k p^*_k)^2,
$$
by stochastic gradient. Similarly to the proof of Theorem \ref{thm1}, one can show that $p^n_k$ converges to $\frac{\nu_k + \delta}{\mu_k}$. With this formulation of the optimization problem, the update equation for the new scheduling algorithm is 
\begin{align*}
p^{n+1}_k = [p^n_k + \delta + \beta^n \bOne_{E^n_k} \sum_{(i',i) \in {\cal H}_k}\Delta Q^{n+1}_{(i',i)}]_{\mathcal{C}},
\end{align*}
which is similar to \eqref{eq:update} with an extra $\delta$-slack. 

We consider the same setting and network parameters as the previous section, and simulate the modified algorithm using $\delta = 0.02$. Figure \ref{fig:sim3} demonstrates a substantial reduction in the queue-length of queue 4 and the delay performance of the algorithm. Similar plots can be obtained for queue-lengths of other queues, which we omit to avoid redundancy.\\ \\
{\bf Time-varying demand and service and bursty arrivals.}
In Section \ref{sec:intro}, we 
mentioned that bursty arrivals as well as time-varying service and arrival rates make estimation of parameters of the system very difficult and often inaccurate, 
and used this reason as a main motivation for designing robust scheduling policies. However, the theoretical results are provided for a time-invariant system with memoryless queues. In this section, we investigate the performance of our proposed algorithm in a time-varying system with bursty arrivals. 

Consider the same DAG structure of previous simulations and the same server rates. We model the burstiness of demand as follows. At each time slot, a batch of $B$ jobs arrive to the system with probability $\lambda/B$. To simulate a time-varying system, we consider two modes of network parameters. In the first mode, arrival rate is $\lambda = 1/5$, and task service rates are
$$\mu_1 = 1, \mu_2 = 4/3,  \mu_3 = 2, \mu_4 = 1/2 \text{ and } \mu_5 = 2/3.$$
In the second mode, arrival rate is $\lambda = 1/6$, and task service rates are
$$\mu_1 = 1/2, \mu_2 = 2,  \mu_3 = 1, \mu_4 = 2/5 \text{ and } \mu_5 = 1.$$
Note that the capacity region of mode $2$ is $\lambda < 3/14$. We simulate a network that changes mode every $T$ time slots. 

Figure \ref{fig:sim4} illustrates the queue-length of the queue 4 versus time when the system has parameters $T = 1000$ 
and $B=5$, and $\delta = 0.02$. As one expects, the bursty time-variant system suffers from larger delay. However, as the simulation shows the queue is still stable, and the gradient algorithm is able to track the changes in the network parameters. 


\subsection{Unstable Network with Generic Service Rates}\label{sec:X}
In this subsection, we first propose a \yuan{natural extension} of the robust algorithm for the case that service rates are generic, 
using an approach that is similar to the design of policy when service rates satisfy Assumption \ref{asmp:factor}. 

Define the allocation vector to be $p = [p_{kj}]$. Similar to before, the algorithm tries to minimize $\sum_{k=1}^K (\nu_k - \sum_{j=1}^J \mu_{kj} p_{kj})^2$ using gradient method. Then, a non-robust update of the allocation vector would be
\begin{align}\label{eq:X1}
p^{n+1}_{kj}  = [ p^n_{kj} + \beta^n \mu_{kj} ( \nu_k - \sum_{j=1}^J \mu_{kj} p^n_{kj})]_{\mathcal{C}}.
\end{align}
A robustified version of the update is 
\begin{align}\label{eq:updateX}
p^{n+1}_{kj} = [p^n_{kj} + \beta^n \bOne_{E^n_k} \sum_{(i',i) \in {\cal H}_k}\Delta Q^{n+1}_{(i',i)}]_{\mathcal{C}}.
\end{align}
We now demonstrate that the \yuan{extension} of \yuan{our} robust algorithm \yuan{need not be} throughput-optimal via a simple example. We consider a specific queueing network known in the literature as ``X" model \cite{BT2011}, shown in Figure \ref{fig:X} with generic service rates $\mu_{kj}$ that cannot be factorized into $2$ factors $\mu_k$ and $\alpha_j$. In our setting, the queueing network is equivalent to having $2$ types of DAGs, each of them consisted of a single task with different service characteristics. There are $2$ servers in the system. The network parameters are as follows.
$$
\lambda_1 = \lambda_2 = 0.3, ~\mu_{11} = \mu_{22} = 1/8, ~\text{and} ~ \mu_{12}=\mu_{21} = 3/8.
$$
\begin{figure}
\centering
    \includegraphics[width= 0.35\textwidth]{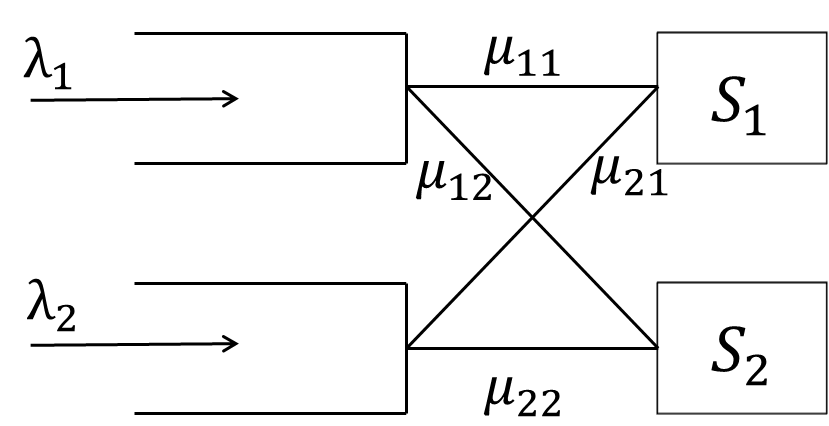}
  \caption{\label{fig:X}The X model} 
\end{figure}
It is easy to check that the stability region is $\lambda_1 \leq 3/8$ and $\lambda_2 \leq 3/8$, which is achieved by server $1$ always working on task $2$ and server $2$ always working on task $1$. Let $p = [p_{11},p_{12},p_{21},p_{22}]$ be the allocation vector for this example. Then, the update is
\begin{align*}
\left[ \begin{array}{cc} 
p^{n+1}_{1j}\\
p^{n+1}_{2j} 
\end{array}
\right] = \left[ \begin{array}{cc} 
p^n_{1j} + \beta^n \bOne_{E^n_1} \Delta Q^{n+1}_1 \\
p^n_{2j} + \beta^n \bOne_{E^n_2} \Delta Q^{n+1}_2
\end{array}
\right]_{\mathcal{C}}  
\end{align*} 
for $j=1,2$. We remark that the update of allocation vector is identical for both servers. Thus, it is expected that the allocation vector converges to $p_{kj} = 1/2$ for all $k \in \{1,2 \}$ and $j \in \{ 1,2 \}$, due to symmetry. The simulation results confirm that our scheduling policy is not stable with these network parameters, since the allocation vector converges to $[1/2,1/2,1/2,1/2]$. In this simulation, we set $p^0_{kj} = 0.1$ for all $k$ and $j$. In general, the reason \yuan{for the convergence} to a \yuan{sub-optimal} allocation vector for generic $\mu_{kj}$ is that the skewed gradient projection (after dropping the term $\mu_{kj}$ from \eqref{eq:X1} to \eqref{eq:updateX}) does not converge, even without noise.

In general, it would be interesting to find out whether there exists a robust scheduling policy that stabilizes the X model. Due to the underlying symmetry of the problem, it is reasonable to conjecture that no myopic queue-size policy\footnote{These are scheduling policies that are only a function of the current queue sizes of the network.} can be throughput-optimal in this example.
\begin{figure}
        \centering
        \begin{subfigure}{0.4\textwidth}
 			\includegraphics[width= 4.8cm, height = 3.6cm]{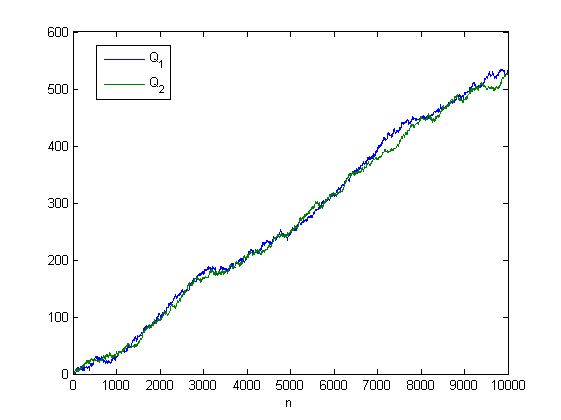}
  \caption{\label{fig:X1}The robust scheduling policy is unstable for the X model.} 
        \end{subfigure}
        \qquad
        \begin{subfigure}{0.4\textwidth}
         \includegraphics[width= 4.8cm, height = 3.6cm]{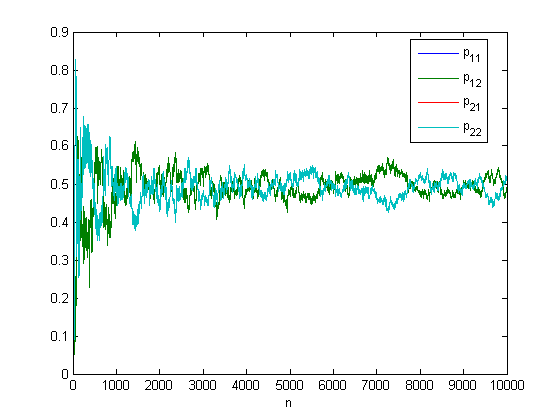}
  		\caption{\label{fig:X2}Allocation vector $p^n$ converges to $[1/2,1/2,1/2,1/2]$.}
        \end{subfigure}        
        \caption{This figure shows the simulation results for the X model.}\label{fig:Xsim}
\end{figure}
